\documentclass[preprint,12pt]{elsarticle}

\usepackage{longtable}
\usepackage{booktabs} 
\usepackage{amsmath, amssymb, amsfonts, amsthm, mathrsfs}

\usepackage{graphicx}
\usepackage{subcaption}
\usepackage{booktabs}
\usepackage{float}
\usepackage{afterpage}
\usepackage[hidelinks]{hyperref}
\usepackage{bookmark}
\usepackage{xcolor}
\usepackage[capitalize]{cleveref}

\usepackage{enumitem}
\usepackage{geometry} 
\usepackage{natbib}

\newtheorem{theorem}{Theorem}[section]
\newtheorem{lemma}[theorem]{Lemma}

\newtheorem{definition}[theorem]{Definition}

\theoremstyle{definition}  
\newtheorem{example}{Example}[section]
\newtheorem{remark}{Remark}
\theoremstyle{plain}  
\newtheorem*{assumption}{Assumption}

\begin{document}

\begin{frontmatter}

\title{Onsager--Machlup Functional for SDE with Time-Varying Fractional Noise\tnoteref{1}}

\author[inst1]{Yanbin Zhu}
\ead{zhuyb23@mails.jlu.edu.cn}
 \author[inst1]{Xiaomeng Jiang\corref{cor1}}
 \ead{jxmlucy@hotmail.com}
 \author[inst1,inst2]{Yong Li}
 \ead{liyong@jlu.edu.cn}
 \address[inst1]{College of Mathematics, Jilin University, Changchun 130012, PR China}
 \address[inst2]{School of Mathematics and statistics and Center for Mathematics and Interdisciplinary Sciences, Northeast Normal University, Changchun 130024, PR China}
 \cortext[cor1]{Corresponding author.}
  \tnotetext[1]{This work is supported by the National Key R\&D Program of China (No.~2023YFA1009200), the National Key Project of the National Natural Science Foundation of China (Grant No.~12531009),  and the National Natural Science Foundation of China (Grant Nos.~12471183, 12071175).}

\begin{abstract}
In this paper, we derive the Onsager–Machlup functional for stochastic differential equations driven by time-varying fractional noise:
\begin{equation*}
      X_t = x_0 + \int_0^t b_s(X_s)\,ds + \int_0^t \sigma_s\,dB^H_s.
\end{equation*}
Our main results are established for $H \in (1/4,1)$ by extending small ball probability estimates and the Girsanov theorem for fractional Brownian motion to the case with time-dependent coefficients. Regarding the choice of norms, for $ \frac{1}{4} < H < \frac{1}{2} $, the analysis is valid under the supremum norm and Hölder norms of order $ 0 < \beta < H - \frac{1}{4} $. For $ \frac{1}{2} < H < 1 $, the analysis applies to Hölder norms of order $\beta$ satisfying $H - \frac{1}{2} < \beta < H - \frac{1}{4}$. In the case $H=1/2$, the admissible norms depend on the spatial regularity of the drift coefficient $b$. Specifically, if $b$ is $n$-times continuously differentiable, then Hölder norms of order $0<\beta<\frac{1}{2}-\frac{1}{2n}$ are permissible. To validate our theoretical findings, we conduct numerical simulations for a classical double-well potential system, illustrating how time-varying fractional noise affects the transition dynamics between metastable states.
\end{abstract}

\begin{keyword}
Onsager–Machlup functional \sep fractional Brownian motion \sep small ball probability estimates \MSC [2020] 60H10 \sep 60G22 \sep 60F10
\end{keyword}

\end{frontmatter}

\allowdisplaybreaks

\section{Introduction}
 Consider the following stochastic differential equation (SDE) driven by time-varying fractional noise:
\begin{equation}
    X_t = x_0 + \int_0^t b_s(X_s)\,ds + \int_0^t \sigma_s\,dB^H_s, \label{fir}
\end{equation}
where $\sigma$ is a deterministic function of time satisfying a non-degeneracy condition, and $B^H = \{B^H_t, t \in [0,1]\}$ is a fractional Brownian motion with Hurst parameter $H \in (0,1)$. 
The integral with respect to the fractional Brownian motion is interpreted in the Wiener sense.

This paper investigates the limiting behavior of the following ratio:
\begin{equation*}
    \frac{\mathbb{P}\big(\Vert X_\cdot - \phi_\cdot \Vert \leq \varepsilon\big)}{\mathbb{P}\big(\Vert \textstyle{\int_0^\cdot} \sigma_s\,dB^H_s \Vert \leq \varepsilon\big)},
\end{equation*}
as $\varepsilon \to 0$. We define the Onsager-Machlup functional by
\begin{equation}
    \exp\big(J(\phi,\dot{\phi})\big) = \lim_{\varepsilon \to 0} \frac{\mathbb{P}\big(\Vert X_\cdot - \phi_\cdot \Vert \leq \varepsilon\big)}{\mathbb{P}\big(\Vert \textstyle{\int_0^\cdot} \sigma_s\,dB^H_s \Vert \leq \varepsilon\big)},\label{omfunctional}
\end{equation}
where $\phi - x_0=K_H^\sigma(\dot{\phi})$, $ \dot{\phi}\in L^2([0,1])$ . The operator $K_H^\sigma$, which depends on both $\sigma$ and $H$, will be defined subsequently.

In this work, we establish a general form of the Onsager–Machlup (OM) functional for SDEs driven by time-varying fractional noise. Our main result can be expressed as
\begin{equation}
    J(\phi, \dot{\phi}) = -\frac{1}{2}
    \int_0^1 \left( \dot{\phi}_s - (K_H^\sigma)^{-1}\!\left( \int_0^\cdot b_u(\phi_u) \, du \right)(s) \right)^2 
    + d_H \, \partial_x b_s(\phi_s) \, ds ,
\end{equation}
where $d_H$ is a constant depending on $H$.  The explicit form of the operator $(K_H^{\sigma})^{-1}$ in our results varies with 
$H \in (1/4, 1/2)$, $H = 1/2$, and $H \in (1/2, 1)$, 
as detailed in Theorems~\ref{result1}, \ref{result2}, and~\ref{result3}.
This result extends the work of Zhang et al.~\cite{zhang2024onsagermachlupfunctionalstochasticdifferential} to the case of fractional noise with time-dependent coefficients.  
Our proof combines a generalized Girsanov transformation with techniques from fractional calculus \cite{Nualart}. By expressing the fractional Brownian motion as an integral with respect to standard Brownian motion,
\begin{equation*}
    \int_0^T \psi_s \, dB^H_s
    = \int_0^T (K_H^* \psi)(s) \, dW_s ,
\end{equation*}
we obtain explicit expressions for the exponential term in the OM functional and derive small ball estimates under non-stationary increments.  
For $H=\tfrac{1}{2}$, we recover stronger results under Hölder norms using Itô’s calculus.

The Onsager–Machlup functional characterizes the most probable paths of diffusion processes and serves as an analogue to the Lagrangian in dynamical systems, describing the optimal evolution trajectory of a particle along a given path. A fundamental challenge is to understand how systems transition between states under stochastic perturbations. Given two states $x_1$ and $x_2$, and considering paths $\phi$ from an appropriate function space connecting $x_1$ to $x_2$, the path that maximizes $\exp(J(\phi,\dot{\phi}))$ can be determined via variational methods. This optimal path represents the most probable transition path from $x_1$ to $x_2$. Thus, the Onsager–Machlup functional provides a deterministic characterization of the most probable transition path through the variational principle, with significant applications in physics and chemistry (see \cite{BATTEZZATI2013163,Durr1979}).

The Onsager–Machlup functional was first introduced by Onsager and Machlup \cite{Onsager1953} in 1953. Ikeda and Watanabe \cite{ikeda1981stochastic} later established rigorous results for $\phi \in C^2([0,1], \mathbb{R}^d)$ under the supremum norm. Moret and Nualart \cite{Nualart} pioneered the study of the Onsager–Machlup functional for SDEs driven by fractional Brownian motion, employing techniques from fractional calculus and a fractional Girsanov transform. 

Recently, research on the Onsager-Machlup functional has attracted considerable attention. Zhang et al.~\cite{zhang2024onsagermachlupfunctionalstochasticdifferential} extended the theory to the case of time-varying Brownian motion, while in~\cite{zhang2025persistenceinvarianttoristochastic} they combined KAM theory with the Onsager-Machlup functional and large deviation principles for the stochastic nonlinear Schrödinger equation on infinite lattices, establishing the persistence and probabilistic stability of low-dimensional invariant tori. Huang et al. \cite{duanjq} derived the Onsager-Machlup functional for jump-diffusion processes using the probability flow approach for Markov processes. Carfagnini and Wang \cite{Carfagnini2024} interpret Loewner energy as the Onsager-Machlup functional for {$\text{SLE}_{\kappa}$} loop measure for any fixed $\kappa\in (0,\frac{1}{4}]$. Li and Li \cite{doi:10.1137/20M1310539} proved that the $\Gamma$-limit of the OM functional on the space of curves is the geometric form of the FW functional in a proper time scale $T=T(\varepsilon)$ as $\varepsilon \to 0$.

The paper is organized as follows. Section~\ref{perliminary} introduces notation and preliminary results. Section~\ref{main results} presents the main theorems and their proofs. Section~\ref{Numeriacal experiments} derives the Euler–Lagrange equation for the OM functional and validates our theory through numerical simulations of a double-well system.

\section*{List of Symbols}
The main notations used for function spaces and norms are summarized below.
\begin{longtable}[htbp]{cl}

\toprule
\textbf{Symbol} & \textbf{Description} \\
\midrule
\endfirsthead

\toprule
\textbf{Symbol} & \textbf{Description} \\
\midrule
\endhead

\midrule
\multicolumn{2}{r}{\textit{Continued on next page}} \\
\endfoot

\bottomrule
\endlastfoot

$L^p([a,b])$ & Space of measurable functions $f$ with \\
& $\displaystyle \|f\|_{L^p} = \left(\int_a^b |f(s)|^p\,ds\right)^{1/p} < \infty$. \\[0.4em]
$C([a,b])$ & Space of continuous functions on $[a,b]$. \\[0.4em]
$C_0([a,b])$ & Subspace of $C([a,b])$ with $f(a) = 0$. \\[0.4em]
$C^\alpha([a,b])$ & Space of $\alpha$-H\"older continuous functions with \\
& $\displaystyle [f]_\alpha = \sup_{s \ne t} \frac{|f_t - f_s|}{|t-s|^\alpha} < \infty$. \\[0.4em]
$C_0^\alpha([a,b])$ & Subspace of $C^\alpha([a,b])$ with $f(a) = 0$. \\[0.4em]
$\Vert \cdot \Vert_{L^p}$ & The $L^p$-norm. \\
$[\cdot]_\alpha$ & H\"older seminorm of order $\alpha$. \\
$\Vert \cdot \Vert_\alpha$ & H\"older norm of order $\alpha$. \\
$\Vert \cdot \Vert_\infty$ & Supremum norm. \\
$\Gamma(\cdot)$ & Gamma function. \\
$\mathrm{B}(\cdot,\cdot)$ & Beta function. \\
\end{longtable}

\section{Preliminaries}  \label{perliminary}
In this section, we recall some basic notations, assumptions, and lemmas that will be used in the sequel.

\subsection{Fractional calculus}

We first introduce the basic concepts of fractional calculus; for further details, refer to \cite{Samko1993}.
 \begin{definition}
   	Let $f\in L^1([a,b])$. The integrals 
   	\begin{align*}
   		(I_{a^+}^\alpha f)(x)&:=\frac{1}{\Gamma(\alpha)}\int_a^x (x-y)^{\alpha-1}f(y)dy,\quad x\geq a,\\
   		(I_{b^-}^\alpha f)(x)&:=\frac{1}{\Gamma(\alpha)}\int_x^b (y-x)^{\alpha-1}f(y)dy,\quad x\leq b,
   	\end{align*}
   	where $\alpha>0$, are respectively called the right and left Riemann--Liouville fractional integrals of order $\alpha$.

   \end{definition}
   
   For any $\alpha\geq 0,$ any $f\in L^p([a,b])$ and $g\in L^q([a,b])$ where $1/p+1/q\leq \alpha,$ we have:
   \begin{equation}
   	\int_a^b f(s)(I_{a^+}^\alpha g)(s)ds=\int _a^b (I_{b^-}^\alpha f)(s) g(s) ds.\label{ffubini}
   \end{equation}
   
   	
   	If $1 \leq p < \infty$, we denote by $I_{a^+}^\alpha(L^p)$ the image of $L^p([a,b])$ under the operator $I_{a^+}^\alpha$. Similarly, $I_{b^-}^\alpha(L^p)$ can be defined.

   	\begin{definition}
   		 Let $f\in I_{a^+}^\alpha(L^p) $, $g\in I_{b^-}^\alpha(L^p) $.
   		 Each of the expressions 
   		 \begin{align*}
   		(D_{a^+}^\alpha f)(x)&:=\left(\frac{d}{dx} \right)^{[\alpha]+1}I_{a^+}^{1+[\alpha]-\alpha}f(x),\\
   		(D_{b^-}^\alpha g)(x)&:=\left(-\frac{d}{dx} \right)^{[\alpha]+1}I_{b^-}^{1+[\alpha]-\alpha}g(x),
   	   	\end{align*}
   		 are respectively called the right and left fractional derivative.
   \end{definition}
  From \eqref{ffubini}, we deduce the formula  
\begin{equation}
  \int_a^b f(s)(D_{a^+}^\alpha g)(s)ds
  =\int _a^b (D_{b^-}^\alpha f)(s) g(s) ds,\quad 0<\alpha<1, \label{fffubini}
\end{equation}
which holds under the assumptions that \( f\in I^\alpha_{b^-}(L^p) \) and \( g\in I^\alpha_{a^+}(L^q) \) satisfy \( 1/p+1/q\leq 1+\alpha \).
   
	If $f\in I_{a^+}^\alpha(L^p) $, the function $ \phi $ such that $f=I_{a^+}^\alpha(\phi)$ is unique in $L^p $. Fractional derivatives can be regarded as the inverse operation of fractional integrals.
 
 When $\alpha p > 1$, any function in $I_{a^+}^\alpha(L^p)$ is 
$(\alpha - \tfrac{1}{p})$-H\"older continuous. Moreover, every H\"older continuous 
function of order $\beta > \alpha$ admits a fractional derivative of order 
$\alpha$; see \cite[Proposition~2.1]{Decreusefond1999}.

Although fractional derivatives are originally defined as derivatives of 
fractional integrals, they also admit an explicit representation in terms of 
Weyl’s formula (see \cite[Remark~5.3]{Samko1993}):
\begin{equation}
D^\alpha_{a^+} f(x) 
= \frac{1}{\Gamma(1-\alpha)} \left( \frac{f(x)}{(x-a)^\alpha} 
+ \alpha \int_a^x \frac{f(x) - f(y)}{(x-y)^{\alpha+1}}\,dy \right), 
\quad 0<\alpha<1,
\label{Weyl}
\end{equation}
where the improper integral converges in the $L^p$ sense. Therefore, if \( f \) has H\"older continuity with an exponent strictly greater than \( \alpha \),  the fractional derivative of order \( \alpha \) exists.

  	For the case of the right fractional derivative, we can also easily prove a similar representation.
    \begin{lemma}
Let $0<\alpha<1$ and $f \in I_{T^-}^\alpha\!\left(L^2[0,T]\right)$. Then
\begin{equation}
D_{T^-}^\alpha f(s) 
= \frac{1}{\Gamma(1-\alpha)} \left(
    \frac{f(s)}{(T-s)^\alpha}
    - \alpha \int_s^T \frac{f(u)-f(s)}{(u-s)^{\alpha+1}}\,du
  \right).
\end{equation}
\end{lemma}

	\begin{proof}
	We restrict ourselves to the case $f \in C^1([0,T])$; situations with lower 
regularity can be treated by mollification. For $0<\alpha<1$, the right fractional derivative can be written as
\begin{align*}
&D^\alpha_{T^-} f(s)\\ 
=& \frac{-1}{\Gamma(1-\alpha)} \frac{d}{ds}\int_s^T (u-s)^{-\alpha} f(u)\,du \\
=& \frac{-1}{\Gamma(1-\alpha)} \frac{d}{ds}\left(
      \int_s^T (u-s)^{-\alpha}\big(f(u)-f(s)\big)\,du
      + f(s)\int_s^T (u-s)^{-\alpha}\,du \right) \\
=& \frac{-1}{\Gamma(1-\alpha)} \Bigg(
     - \lim_{u\to s} (u-s)^{-\alpha}\big(f(u)-f(s)\big) \\
    &\qquad\qquad\quad+ \int_s^T \alpha (u-s)^{-\alpha-1}\big(f(u)-f(s)\big)\,du \\
&\qquad\qquad\quad - \int_s^T (u-s)^{-\alpha} f'(s)\,du
     + \frac{(T-s)^{1-\alpha}}{1-\alpha} f'(s)
     - (T-s)^{-\alpha} f(s) \Bigg) \\
=& \frac{1}{\Gamma(1-\alpha)} \Bigg(
     \frac{f(s)}{(T-s)^\alpha}
     - \alpha \int_s^T \frac{f(u)-f(s)}{(u-s)^{\alpha+1}}\,du
   \Bigg).
\end{align*}

	\end{proof}

\subsection{Fractional Stochastic integration with respect to fractional Brownian motion}
      Let $W=\{W_t,t\in [0,1] \}$ be a Wiener process defined in the canonical probability space $(\Omega,\mathcal{F},\mathbb{P})$, where 
 $\Omega=C_0([0,1])$
   and $\mathbb{P}$ is the Wiener measure.
	A real-valued continuous process $\{B^H_t,\ t\in [0,T]\}$ is called a fractional Brownian motion with Hurst parameter $H\in (0,1)$ if it is a centered Gaussian process with covariance function  
\[
  \mathbb{E}[B^H_tB^H_s]
  = \tfrac{1}{2}\big(|t|^{2H}+|s|^{2H}-|t-s|^{2H}\big)
  := R_H(t,s).
\]
For simplicity, we assume through all that $\alpha = |H - 1/2|$.

When $H=\tfrac{1}{2}$, fractional Brownian motion reduces to the classical Brownian motion.  
However, if $H\neq \tfrac{1}{2}$, it is neither a semi-martingale nor a Markov process. Moreover, its sample paths belong to $C^{H-\varepsilon}_0([0,T])$ $P$-a.s. for every $\varepsilon>0$.

It is well known that an fractional Brownian motion $B^H_t$ admits a Wiener integral representation: there exists a deterministic kernel $K_H(t,s)$ and a standard Brownian motion $W$ such that  
\[
  B^H_t = \int_0^t K_H(t,s)\,dW_s,
\]
where
\[
  K_H(t,s)=
  \begin{cases}
    c_H\, s^{-\alpha}\displaystyle\int_s^t (u-s)^{\alpha-1}u^\alpha\,du, & H>\tfrac{1}{2}, \\[2ex]
    b_H\Big[(t/s)^{-\alpha}(t-s)^{-\alpha}
    +\alpha s^\alpha \int_s^t (u-s)^{-\alpha}u^{-(\alpha+1)}\,du\Big], & H<\tfrac{1}{2},
  \end{cases}
\]
with constants
\[
  c_H=\sqrt{\frac{H(2H-1)}{\mathrm{B}(2-2H,H-\tfrac{1}{2})}}, 
  \qquad 
  b_H=\sqrt{\frac{2H}{(1-2H)\mathrm{B}(1-2H,H+\tfrac{1}{2})}}.
\]

The operator \( K_H \), mapping from \( L^2([0,1]) \) to \( I_{0^+}^{H+\frac{1}{2}}(L^2([0,1])) \), is given by:  
\begin{equation}
	(K_Hh)(t) = \int_0^t K_H(t,s) h(s) \, ds.\label{kh operator}
\end{equation}

 From \cite[Lemma 10]{Nualart}, the operator \( K_H \) can be expressed using fractional integrals as follows:   
\[
(K_H h)(s) = \begin{cases}
\displaystyle
I^{1-2\alpha}_{0^+} \, s^{\alpha} I^{\alpha}_{0^+} \, s^{-\alpha} h, & H <1/2, \\
\displaystyle
I^{1}_{0^+} \, s^{\alpha} I^{\alpha}_{0^+} \, s^{-\alpha} h, & H >1/2.
\end{cases}
\]  
 The inverse operator \((K_H)^{-1}\) is then defined as:  
\begin{align}  
\notag (K_H)^{-1} h &= s^{\alpha} D^{\alpha}_{0^+} \, s^{-\alpha} D^{1-2\alpha}_{0^+} h, \quad H < 1/2, \\  
(K_H)^{-1} h &= s^{\alpha} D^{\alpha}_{0^+} \, s^{-\alpha} h', \quad H > 1/2,\label{g1/2}   
\end{align}    
for all \(h \in I^{H+\frac{1}{2}}_{0^+}(L^2)\). If \(h\) is differentiable, the operator simplifies to:  
\begin{equation} \label{l1/2}  
(K_H)^{-1} h = s^{-\alpha} I^{\alpha}_{0^+} \, s^{\alpha} h', \quad H <1/2.  
\end{equation}  
Moreover, the operator \((K_H)^{-1}\) preserves the adaptability property.

Stochastic integrals of deterministic functions with respect to a Gaussian process are called Wiener integrals.  Stochastic integrals with respect to fractional Brownian motion can be defined by  its Gaussianity. For the step function $I_{[0,t]}(\cdot)$, we assign the inner product 
\[
\langle I_{[0,t]}(\cdot), I_{[0,s]}(\cdot) \rangle_H = R_H(t,s).
\]
We denote by $\mathcal{H}$ the space obtained by completing the step functions under the above inner product.
Then the mapping $\mathcal{I}:I_{[0,t]}\longmapsto B^H_t$ corresponds to an isometry between the step function under the inner product $\langle  \cdot, \cdot \rangle_H $ and the Gaussian variable in $L^2(\mathbb{P})$. By extending this isometry, we obtain our integral $\mathcal{I}$ as an isometry from $\mathcal{H}$ to the space of Gaussian random variables in $L^2(\mathbb{P})$.

In order to illustrate the elements of $\mathcal{H}$ and the relation between the fractional Brownian  integral and the Brownian motion integral, we first introduce the following operator:
\[
  (K_H^*f)(s)=
  \begin{cases}
    c_H\Gamma(\alpha)\, s^{-\alpha}I_{1-}^\alpha\big(s^\alpha f_s\big), & H>\tfrac{1}{2}, \\[1ex]
    b_H\Gamma(1-\alpha)\, s^\alpha D^\alpha_{1^-}\big(s^{-\alpha}f_s\big), & H<\tfrac{1}{2}.
  \end{cases}
\]

\begin{theorem}[\cite{Biagini2008}]
Let $H\in(0,1)$. 
If $\psi\in\mathcal{H}$, we have
\begin{equation*}
  \int_0^T \psi_s\,dB^H_s 
  := \int_0^T (K_H^*\psi)(s)\,dW_s.
\end{equation*}
\end{theorem}
In particular, for $\psi_s=\mathbf{1}_{[0,t]}(s)$, we recover
\begin{equation*}
  B^H_t=\int_0^t K_H(t,s)\,dW_s.
\end{equation*}

Analogous to the It\^o isometry, we have the following isometry for stochastic integrals with respect to fractional Brownian motion.
\begin{lemma}[\cite{Biagini2008}]\label{isometry}
If $f,g\in\mathcal{H}$, then
\begin{equation*}
  \mathbb{E}\!\left(\int_0^T f_s\,\mathrm{d}B^H_s \int_0^T g_s\,\mathrm{d}B^H_s\right)
  =\int_0^T (K_H^*f)(s)\,(K_H^*g)(s)\,\mathrm{d}s.
\end{equation*}
For $H > \tfrac{1}{2}$, we have a more intuitive equality:
\begin{equation}
    \mathbb{E}\!\left(\int_0^T f_s\,\mathrm{d}B^H_s \int_0^T g_s\,\mathrm{d}B^H_s\right)
    =H(2H-1)\int_0^T\int_0^T f_t g_s |s-t|^{2H-2}\,\mathrm{d}s\,\mathrm{d}t.
\end{equation}
\end{lemma}

	In this paper, we assume that $\sigma$ is H\"older continuous of order $\gamma$ with $\gamma+H>1$. Under this condition, the stochastic integral
\[
\int_0^1 \sigma_u \, dB^H_u
\]
is well defined pathwise in the sense of Young integration(see Lemma~\ref{young}), and this definition agrees with the Wiener integral.
 
	\begin{lemma}[\cite{Young1936}]   \label{young}
		For $f\in C^\beta([0,1]), g\in C^\gamma([0,1])$, if $\beta+\gamma>1,$ 
		\begin{equation*}
			\int_0^t \sigma_s dg_s,\quad \int_0^t g_sd\sigma_s 
		\end{equation*}
		are well-defined. Furthermore, we have
		\begin{equation*}
			\int_0^t \sigma_s dg_s=\sigma_t g_t-\sigma_0 g_0-\int_0^t g_sd\sigma_s.
		\end{equation*}
	\end{lemma}
	
	\begin{lemma}[\cite{Young1936}] \label{young2}
Let $f,g:[0,T]\to \mathbb{R}$ be functions such that $f \in C^{\beta}$ 
and $g \in C^{\gamma}$ with $\beta,\gamma \in (0,1)$ and $\beta+\gamma>1$. 
Then  for any $0\leq s<t\leq T$, one has
\begin{equation}
	\left|\int_s^t f_r \, dg_r - f_s (g_t-g_s)\right|
   \leq C_{\alpha,\beta}\, [f]_{\beta}\,[g]_{\gamma}\, |t-s|^{\beta+\gamma}.\label{youngint}
\end{equation}
\end{lemma}

\begin{proof}
	For simplicity, it suffices to prove the reduced case $s=0, t=T$. 
We first divide the interval $[0,T]$ into $2^n$ equal parts and set  
\[
   S_n=\sum_{i=0}^{2^n-1} f_{\tfrac{iT}{2^n}}\left(g_{\tfrac{(i+1)T}{2^n}}-g_{\tfrac{iT}{2^n}}\right).
\]
Then,\begin{align*}
   |S_{n+1}-S_n|
   &=\left|\sum_{i=0}^{2^n-1} f_{\tfrac{2iT}{2^{n+1}}}\left(g_{\tfrac{(2i+1)T}{2^{n+1}}}-g_{\tfrac{2iT}{2^{n+1}}}\right)
     + f_{\tfrac{(2i+1)T}{2^{n+1}}}\left(g_{\tfrac{(2i+2)T}{2^{n+1}}}-g_{\tfrac{(2i+1)T}{2^{n+1}}}\right) \right.\\
   &\quad -\left. f_{\tfrac{iT}{2^n}}\left(g_{\tfrac{(i+1)T}{2^n}}-g_{\tfrac{iT}{2^n}}\right)\right| \\
   &=\left|\sum_{i=0}^{2^n-1} \left(f_{\tfrac{(2i+1)T}{2^{n+1}}}-f_{\tfrac{iT}{2^n}}\right)\left(g_{\tfrac{(2i+2)T}{2^{n+1}}}-g_{\tfrac{(2i+1)T}{2^{n+1}}}\right)\right| \\
   &\leq \sum_{i=0}^{2^n-1}[f]_\beta [g]_\gamma \frac{T^{\beta+\gamma}}{2^{(n+1)(\beta+\gamma)}} \\
   &= \frac{[f]_\beta [g]_\gamma}{2^{\beta+\gamma}}\cdot \frac{T^{\beta+\gamma}}{2^{n(\beta+\gamma-1)}}.
\end{align*}

Hence
\begin{align*}
     \left|\int_0^T f_s\,dg_s\right|
  & \leq |S_0| + \frac{[f]_\beta [g]_\alpha}{2^{\beta+\gamma}}\sum_{n=0}^\infty \frac{T^{\beta+\gamma}}{2^{n(\beta+\gamma-1)}}\\
   &\leq |f_0(g_1-g_0)| + \frac{[f]_\beta [g]_\gamma}{2^{\beta+\gamma}}\cdot \frac{T^{\beta+\gamma}}{1-\tfrac{1}{2^{\beta+\gamma-1}}}.
\end{align*}

\end{proof}
    \begin{remark}
        From the above theorem or the Kolmogorov-Chentsov theorem, it also follows that the paths of the integral of fractional Brownian motion are almost surely H\"older continuous with exponent $H - \varepsilon$.
    \end{remark}

 \begin{assumption}[A]\label{ass:A}
We impose the following assumptions on the coefficients of \eqref{fir}.

\textbf{(1) General conditions.}
The coefficients \(b\) and \(\sigma\) satisfy:
\begin{enumerate}
    \item $\sigma \in C^1([0,1])$ and $\inf_{0 \leq s \leq 1} |\sigma_s| > 0$. 
    Without loss of generality, we may assume that $0 < m \leq \sigma \leq M$;
    \item $b$ is continuous in $(t,x)$, twice continuously differentiable with respect to $x$, and bounded.
\end{enumerate}

\textbf{(2) Additional conditions depending on the Hurst parameter.}
\begin{itemize}
    \item[\textbf{(A1)}] \textbf{Case \(1/4 < H < 1/2\):}  
    No further restriction is imposed on \(b\) or \(\sigma\).

    \item[\textbf{(A2)}] \textbf{Case \(1/2 < H < 1\):}  
    In addition to the general conditions, assume that:
    \begin{enumerate}
        \item $b$ satisfies the Lipschitz condition
        \[
           |b_s(x) - b_t(y)| \leq L \big( |t-s| + |x-y| \big),\quad t,s\in[0,1];
        \]
        \item The following inequality holds:
        \[
           1 < \frac{m^2 (2\beta+1)\Gamma(1+\beta)^2}
                    {M^2 \alpha^2 L^2 \Gamma(\beta-\alpha)^2}.
        \]
    \end{enumerate}

    \item[\textbf{(A3)}] \textbf{Case \(H = 1/2\):}  
    In addition to the general conditions, assume that:
    \begin{enumerate}
        \item $b$ possesses continuous partial derivatives with respect to $x$ up to order $n \geq 2$. 
        Moreover, for each $1 \leq k \leq n-1$, the derivative $(\partial_x)^k b_t(x)$ 
        is continuously differentiable in $t$.
    \end{enumerate}
\end{itemize}
\end{assumption}

\begin{remark}
The boundedness of $b$ is required for the application of Girsanov’s theorem. 
In the computation of the Onsager--Machlup functional, $b$ and its derivatives 
will automatically remain bounded under the condition 
\[
\Big\| \int_0^\cdot \sigma_s \, dB^H_s \Big\| \leq \varepsilon.
\]
\end{remark}

\begin{remark}
Condition~(2) in (A2) is imposed to ensure the validity of Girsanov’s theorem. 
If Condition~(2) fails while all other conditions in (A2) remain in force, 
the Onsager--Machlup functional can still be derived on the interval $[0,T]$, 
provided that $T$ satisfies
\[
    T^{2H+1} 
    < 
    \frac{m^2 (2\beta+1)\Gamma(1+\beta)^2}
         {M^2 \alpha^2 L^2 \Gamma(\beta-\alpha)^2},
\]
where $\beta$ denotes the exponent corresponding to the H\"older norm of the noise, 
and $\alpha = |H - \tfrac{1}{2}|$.
\end{remark}

\begin{remark}
    The existence and uniqueness of the solution to \eqref{fir} can be obtained by combining Girsanov's theorem with pathwise uniqueness, according to the Yamada--Watanabe theorem.
\end{remark}

    \subsection{Approximate limits in Wiener space}
    For the Onsager-Machlup functional of fractional Brownian motion, 
the measurable norm and trace theory serve as powerful tools for studying 
the asymptotic behavior of exponential conditional expectations. 
The following content is adapted from $\cite{Nualart}$.

   Let $H^1$ be the Cameron-Martin space:
   \begin{equation*}
   	H^1:=\{h\in \Omega,\text{ $h$ is absolutely continuous and $h'$}\in L^2([0,1])\}.
   \end{equation*}

    Let $Q : H^1 \to H^1$ be an orthogonal projection such that 
$\dim(QH^1) < \infty$. Then $Q$ can be represented as
\[
Qh = \sum_{i=1}^n \langle h, h_i \rangle h_i,
\]
where $(h_1, \ldots, h_n)$ is an orthonormal sequence in $QH^1$. 
We may also define the $H^1$-valued random variable
\[
QW = \sum_{i=1}^n \left( \int_0^1 h_i(s)\, dW_s \right) h_i.
\]
Note that $QW$ does not depend on the particular choice of 
$(h_1, \ldots, h_n)$.

A sequence of orthogonal projections $\{Q_n\}$ on $H^1$ is called an 
\emph{approximating sequence of projections} if $\dim(Q_n H_1) < \infty$ 
and $Q_n$ converges strongly and monotonically to the identity operator on $H^1$.

\begin{definition}
A semi-norm $N$ on $H^1$ is called a \emph{measurable semi-norm} if there 
exists a random variable $\tilde{N} < \infty$ almost surely such that 
for every approximating sequence of projections $\{Q_n\}$ on $H^1$, 
the sequence $N(Q_n W)$ converges in probability to $\tilde{N}$, and
\[
\mathbb{P}(\tilde{N} \leq \varepsilon) > 0 
\quad \text{for all } \varepsilon > 0.
\]
If, in addition, $N$ is a norm on $H^1$, then it is called a 
\emph{measurable norm}.
\end{definition}

We will make use of the following result on measurable semi-norms.

\begin{lemma} \label{measurablenorm}
Let $\{N_n\}$ be a nondecreasing sequence of measurable semi-norms. 
Suppose that
\[
\tilde{N} := \mathbb{P}\text{-}\lim_{n \to \infty} N_n(W)
\]
exists and satisfies $\mathbb{P}(\tilde{N} \leq \varepsilon) > 0$ 
for all $\varepsilon > 0$. Then, if the limit $N(h) = \lim_{n \to \infty} N_n(h)$
exists for all $h \in H^1$, it defines a measurable 
semi-norm.
\end{lemma}

	\begin{theorem}\label{estimate1}
Let \( N \) be a measurable norm on \( H^{1} \). Then, 
\[
\lim _{\varepsilon \to 0} \mathbb{E}\left( \exp \left( \int_{0}^{1} h(s) \, d W_{s} \right) \mid \tilde{N} < \varepsilon \right) = 1,
\]
for all \( h \in L^{2}([0,1]) \).
\end{theorem}
    We recall that an operator $K : L^2([0,1]) \to L^2([0,1])$ is \emph{nuclear} iff 
\[
\sum_{n=1}^\infty \big| \langle K e_n, g_n \rangle \big| < \infty,
\]
for all orthonormal sequences $B=(e_n)$, $B'=(g_n)$ in $L^2([0,1])$.  
 The trace of a nuclear operator $K$ is defined by
\[
\operatorname{Tr} K = \sum_{n=1}^\infty \langle K e_n, e_n \rangle,
\]
for any orthonormal sequence $B=(e_n)$ in $L^2([0,1])$.  
The definition is independent of the sequence we have chosen.

Given a symmetric function $f \in L^2([0,1]^2)$, the Hilbert--Schmidt operator 
$K(f) : L^2 \to L^2$ associated with $f$ is defined by
\begin{equation}\label{eq:Kf}
  (K(f)h)(t) = \int_0^t f(t,u) h(u)\,du.
\end{equation}
The operator $K(f)$ is nuclear iff
\[
\sum_{n=1}^\infty \big|\langle K e_n, e_n \rangle \big| < \infty
\quad \text{for all orthonormal sequences } (e_n) \subset L^2([0,1]).
\]

If $f$ is continuous and $K(f)$ is nuclear, we can compute its trace as follows:
\[
\operatorname{Tr} f := \operatorname{Tr} K(f) = \int_0^1 f(s,s)\,ds.
\]

	\begin{theorem}\label{estimate2}
Let \( f \) be a symmetric function in \( L^{2}([0,1]^{2}) \) and let \( N \) be a measurable norm. If \( K(f) \) is nuclear, then 
\[
\lim _{\varepsilon \to 0} E\left( \exp \left( \int_{0}^{1} \int_{0}^{1} f(s, t) \, d W_{s} d W_{t} \right) \mid \tilde{N} < \varepsilon \right) = e^{-\text{Tr}(f)}.
\]
\end{theorem}
In the classical results, it is known that for fractional Brownian motion, 
both the supremum norm and the H\"older norm are measurable norms. 
We now extend this result to the integral of fractional Brownian motion. 
To this end, it is necessary to first establish a small-ball probability estimate 
for the integral of fractional Brownian motion.

To this end, we first derive an upper bound for its second moment.

\begin{lemma}\label{variance}
	Let \(\sigma\) satisfy Assumption \((A)\).
Define
\[
	U_t^H = \int_0^t \sigma_s \, dB^H_s, 
	\qquad 
	V_h^t = \mathbb{E}\big[(U_{t+h}^H-U_t^H)^2\big].
\]
There exists a constant $c_*>0$ such that 
\begin{equation*}
	\frac{h^{2H}}{V_h^t} \geq c_*, 
	\quad \forall\, 0 \leq t,h \leq 1,\; t+h \leq 1.
\end{equation*}
\end{lemma}
\begin{proof}
Our proof is based on Lemma~\ref{isometry}.
	
For $H > \tfrac{1}{2}$, we have
\begin{align*}
	\mathbb{E}\big[(U_{t+h}^H-U_t^H)^2\big]
	&= H(2H-1)\int_t^{t+h}\!\!\int_t^{t+h} \sigma_r\sigma_s\,|s-r|^{2H-2}\,ds\,dr \\
	&\leq C h^{2H}.
\end{align*}

For $H < \tfrac{1}{2}$,
\begin{align*}
	\mathbb{E}\big[(U_{t+h}^H-U_t^H)^2\big]
	&= \int_0^1 \big\{b_H\Gamma(1-\alpha)s^\alpha D^\alpha_{1^-}\!\big(s^{-\alpha}\sigma_s I_{[t,t+h]}(s)\big)\big\}^2 ds,
\end{align*}
where
\begin{align*}
	&D^\alpha_{1^-}(s^{-\alpha}\sigma_s I_{[t,t+h]}(s))\\
	=& \frac{1}{\Gamma(1-\alpha)}\Bigg(
	\frac{s^{-\alpha}\sigma_s I_{[t,t+h]}(s)}{(1-s)^\alpha}
	- \alpha\int_s^1 \frac{u^{-\alpha}\sigma_u I_{[t,t+h]}(u)-s^{-\alpha}\sigma_s I_{[t,t+h]}(s)}{(u-s)^{\alpha+1}}\,du
	\Bigg).
\end{align*}
Moreover,
\[
	\left|\frac{\sigma_u-\sigma_s}{(u-s)^{\alpha+1}}\right|
	\leq \frac{[\sigma]_\gamma}{(u-s)^{\alpha+1-\gamma}}.
\]
Therefore,
\begin{align*}
	&\mathbb{E}\big[(U_{t+h}^H-U_t^H)^2\big]\\
	\leq & C\Bigg(
	\int_t^{t+h} s^{2\alpha}\Bigg(\frac{s^{-\alpha}\sigma_s}{(t+h-s)^\alpha}
	- \alpha\int_s^{t+h}\frac{u^{-\alpha}\sigma_u-s^{-\alpha}\sigma_s}{(u-s)^{\alpha+1}}\,du\Bigg)^2 ds \\
	&\quad + \int_0^t s^{2\alpha}\Bigg(\int_t^{t+h}(u-s)^{-\alpha-1}u^{-\alpha}\sigma_u\,du\Bigg)^2 ds
	\Bigg) \\
	\leq & Ch^{2H}
	+ C\int_t^{t+h}(t+h-s)^{2(\gamma-\alpha)}\,ds
	+ C\int_0^t\big((t+h-s)^{-\alpha}-(t-s)^{-\alpha}\big)^2 ds \\
	\leq & 2Ch^{2H} + C h^{2(\gamma-\alpha)+1}
	+ C\int_0^t\big((h+x)^{-\alpha}-x^{-\alpha}\big)^2 dx.
\end{align*}
Finally,
\begin{align*}
	&\int_0^t \big((h+x)^{-\alpha}-x^{-\alpha}\big)^2 dx \\
	=& \int_0^{t-h}\big((h+x)^{-\alpha}-x^{-\alpha}\big)^2 dx
	+ \int_{t-h}^t \big((h+x)^{-\alpha}-x^{-\alpha}\big)^2 dx \\
	\leq & C\int_h^1 \frac{h^2}{(x+h)^{2\alpha}x^2} dx
	+ \int_0^h \frac{h^{2\alpha}}{(x+h)^{2\alpha}x^{2\alpha}} dx \\
	\leq & C(h^{1-2\alpha}+h^{2-2\alpha}) \\
	\leq & Ch^{2H}.
\end{align*}
Hence
\[
	\frac{h^{2H}}{V_h^t} \geq c^*.
\]
\end{proof}

    Before proving the small ball probability, we first present two lemmas.
	\begin{lemma}
Let $Z \sim N(0,1)$. Then the following bounds hold:
\begin{equation}
    \mathbb{P}(|Z| \le t) \ge \frac{t}{2}, \quad \forall\, 0 \le t \le 1,  \label{smpexp1}
\end{equation}
and
\begin{equation}
    \mathbb{P}(|Z| \le st) \ge \exp\!\left( -\theta_s e^{-(st)^2/2} \right), 
    \quad s > 0,\ t \ge 1,\label{smpexp2}
\end{equation}
where $\theta_s = \left( 1 - e^{-s^2/2} \right)^{-1}$.
\end{lemma}

\begin{lemma}[{\cite{Sidak1968}}]
If $(Z_1,\dots,Z_k)$ has a normal distribution with mean $0$ 
and an arbitrary correlation matrix, then 
\begin{equation}\label{Sidak}
    \mathbb{P}\!\left(|Z_1|<c_1, \dots, |Z_k|<c_k\right) 
    \;\geq\; \prod_{i=1}^k \mathbb{P}(|Z_i|<c_i).
\end{equation}  
\end{lemma}

\begin{lemma}\label{smp}
	Let $\{U_t,\,0 \leq t \leq 1\}$ be a separable centered Gaussian process with $U_0 = 0$. Define  
\begin{equation*}
	V_h^t = \mathbb{E}\bigl[(U_{t+h} - U_t)^2\bigr].
\end{equation*}
Assume that there exists a concave function $\sqrt{V_\cdot}$ such that  
\begin{equation*}
	\frac{\sqrt{V_h}}{\sqrt{V_h^t}} \geq c_*, 
	\quad \forall\, 0 \leq t,h \ \text{with}\ t+h \leq 1,
\end{equation*}
and let $f \in C_0[0,1]$ be a nondecreasing function that is strictly positive on $(0,1]$.  
Moreover, suppose that for some $\beta > 0$, the function  
\[
\varepsilon \mapsto \frac{\sqrt{V_\varepsilon}}{\varepsilon^\beta f(\varepsilon)}
\]
is nondecreasing on $(0,1]$.  

Then there exists a constant $M_\beta > 0$, depending only on $\beta$ and $c_*$, such that  
\begin{equation*}
	\mathbb{P}\!\left(
		\sup_{0 \leq s \neq t \leq 1}
		\frac{|U_t - U_s|}{f(|t-s|)}
		\leq \frac{\sqrt{V_\varepsilon}}{f(\varepsilon)} 
	\right) 
	\geq \exp\!\left(-\frac{M_\beta}{\varepsilon}\right).
\end{equation*}
\end{lemma}

\begin{proof}
Our method follows the approach of \cite{KuelbsLiShao1995}. 
By time discretization, we obtain the following lower bound by inequality~\eqref{Sidak}:
\begin{align*} 
&\mathbb{P}\!\left( 
   \sup_{0 \leq s \neq t \leq 1} 
   \frac{|U_t - U_s|}{f(|t-s|)} 
   \leq \frac{\sqrt{V_\varepsilon}}{f(\varepsilon)} 
 \right) \\
 \geq & \mathbb{P}\Biggl(
   \max_{1 \leq i \leq 2^l} 
   |U_{i/2^l} - U_{(i-1)/2^l}| \leq x_l, \ \forall\, l \geq 1, \\
&\qquad  
   \max_{0 \leq i \leq 2^j/\varepsilon}
   \max_{1 \leq m \leq 2^{l-j}}
   \frac{|U_{(m+1)\varepsilon/2^l + i\varepsilon/2^j} 
          - U_{m\varepsilon/2^l + i\varepsilon/2^j}|}
        {f(\varepsilon/2^{j+1})} 
   \leq y_{j,l},  \\ 
  &\qquad \forall\, l \geq j+1,\ j \geq 0 
 \Biggr) \\
 \geq & A \cdot B, 
\end{align*} 
where 
\begin{equation*} 
A = \prod_{l=1}^\infty \prod_{i=1}^{2^l} 
    \mathbb{P}\!\left(
      |U_{i/2^l} - U_{(i-1)/2^l}| \leq x_l
    \right), 
\end{equation*} 
and 
\begin{equation*} 
B = \prod_{j=0}^\infty 
    \prod_{l=j+1}^\infty 
    \prod_{m=1}^{2^{l-j}} 
    \prod_{i=0}^{2^j/\varepsilon} 
    \mathbb{P}\!\left(
      \frac{|U_{(m+1)\varepsilon/2^l + i\varepsilon/2^j} 
             - U_{m\varepsilon/2^l + i\varepsilon/2^j}|}
           {f(\varepsilon/2^{j+1})} 
      \leq y_{j,l}
    \right). 
\end{equation*}  
Here 
\begin{align*} 
x_l &= \frac{1 - 2^{-\beta/2}}{4} 
       \sqrt{V\!\left( (3/2)^{-|l-n_\varepsilon|}\varepsilon \right)}, \\ 
y_{j,l} &= \frac{2^{\beta(l-j)/2} \, 2^{j\beta} \, (1 - 2^{-\beta/2}) 
          \sqrt{V(\varepsilon 2^{-l})}}
         {3 f(\varepsilon 2^{-j-1})}. 
\end{align*} 

There exits a constant $n_\varepsilon$ such that 
\begin{equation*} 
1/\varepsilon\leq 2^{n_\varepsilon} \leq 2/\varepsilon. 
\end{equation*} 
Then 
\begin{align*} 
A &= \prod_{l=1}^\infty \prod_{i=1}^{2^l} \mathbb{P}\big(|U_{i/2^l}-U_{(i-1)/2^l}|\leq x_l\big)\\ 
&= \prod_{l=1}^\infty \prod_{i=1}^{2^l} \mathbb{P}\Big(\frac{|U_{i/2^l}-U_{(i-1)/2^l}|}{\sqrt{V_{1/2^l}^{(i-1)/2^l}}} \leq \frac{x_l}{\sqrt{V_{1/2^l}^{(i-1)/2^l}}} \Big)\\ 
&= \prod_{l=1}^\infty \prod_{i=1}^{2^l} \mathbb{P}\Big(\frac{|U_{i/2^l}-U_{(i-1)/2^l}|}{\sqrt{V_{1/2^l}^{(i-1)/2^l}}} \leq \frac{1-2^{-\beta/2}}{4} \frac{\sqrt{V_{\varepsilon (2/3)^{|l-n_\varepsilon|}}}}{\sqrt{V_{1/2^l}^{(i-1)/2^l}}} \Big)\\ 
&\geq \prod_{l=1}^\infty \prod_{i=1}^{2^l} \mathbb{P}\Big(\frac{|U_{i/2^l}-U_{(i-1)/2^l}|}{\sqrt{V_{1/2^l}^{(i-1)/2^l}}} \leq c_* \frac{1-2^{-\beta/2}}{4} \frac{\sqrt{V_{\varepsilon (2/3)^{|l-n_\varepsilon|}}}}{\sqrt{V_{1/2^l}}} \Big)\\ 
&= \prod_{l=1}^\infty \mathbb{P}\Big(|Z| \leq c_* \frac{1-2^{-\beta/2}}{4} \frac{\sqrt{V_{\varepsilon (2/3)^{|l-n_\varepsilon|}}}}{\sqrt{V_{1/2^l}}} \Big)^{2^l} \\ 
&= A_1 \cdot A_2, 
\end{align*} 
where 
\begin{equation*} 
A_1 = \prod_{l=1}^{n_\varepsilon} \mathbb{P}\Big(|Z|\leq c_* \frac{1-2^{-\beta/2}}{4} \frac{\sqrt{V_{\varepsilon(2/3)^{|l-n_\varepsilon|}}}}{\sqrt{V_{1/2^l}}} \Big)^{2^l},
\end{equation*} 
and 
\begin{equation*} 
A_2 = \prod_{l=n_\varepsilon+1}^\infty \mathbb{P}\Big(|Z|\leq c_* \frac{1-2^{-\beta/2}}{4} \frac{\sqrt{V_{\varepsilon(2/3)^{|l-n_\varepsilon|}}}}{\sqrt{V_{1/2^l}}} \Big)^{2^l}. 
\end{equation*} 
For the $A_1$ term,
\begin{align} 
A_1 &= \prod_{l=1}^{n_\varepsilon} \mathbb{P}\Big(|Z|\leq c_* \frac{1-2^{-\beta/2}}{4} \frac{\sqrt{V_{\varepsilon(2/3)^{|l-n_\varepsilon|}}}}{\sqrt{V_{1/2^l}}}\Big)^{2^l} \notag\\ 
&= \prod_{l=1}^{n_\varepsilon} \mathbb{P}\Big(|Z|\leq c_* \frac{1-2^{-\beta/2}}{4} \frac{\sqrt{V_{\varepsilon(2/3)^{(n_\varepsilon-l)}}}}{\sqrt{V_{1/2^l}}}\Big)^{2^l} \notag\\ 
&\geq \prod_{l=1}^{n_\varepsilon} \mathbb{P}\Big(|Z|\leq c_* \frac{1-2^{-\beta/2}}{4} \frac{\sqrt{V_{2^{-n_\varepsilon}(2/3)^{(n_\varepsilon-l)}}}}{\sqrt{V_{1/2^l}}}\Big)^{2^l} \notag\\ 
&= \prod_{l=1}^{n_\varepsilon} \mathbb{P}\Big(|Z|\leq c_* \frac{1-2^{-\beta/2}}{4} \frac{\sqrt{V_{3^{-(n_\varepsilon-l)}2^{-l}}}}{\sqrt{V_{1/2^l}}}\Big)^{2^l} \notag\\ 
&\geq \prod_{l=1}^{n_\varepsilon} \mathbb{P}\Big(|Z|\leq c_* \frac{1-2^{-\beta/2}}{4}\,3^{-(n_\varepsilon-l)}\Big)^{2^l} \label{conve}\\ 
&\geq \prod_{l=1}^{n_\varepsilon} \left(\tfrac{c_*^1}{2}\,3^{-(n_\varepsilon-l)}\right)^{2^l} \label{smpex1}\\ 
&= \exp\!\Big(-\sum_{l=1}^{n_\varepsilon} 2^l \{-\ln c_*^1+(n_\varepsilon-l)\ln3 \}\Big) \notag\\ 
&= \exp\!\Big(2(2^{n_\varepsilon}-1)\ln c_*^1 + (2(n_{\varepsilon}+1)-2^{n_{\varepsilon}+1})\ln3 \Big) \notag\\ 
&\geq \exp\!\Big(-2^{n_\varepsilon} \{2(1-2^{-n_\varepsilon})\ln(1/c_*^1) +(2-2(n_{\varepsilon}+1)/2^{n_\varepsilon})\ln3 \}\Big) \notag\\ 
&\geq \exp\!\Big(-\tfrac{1}{\varepsilon}\,4\ln\!\Big(\tfrac{3}{c_*^1}\Big) \Big), \notag 
\end{align} 
where 
\[ c_*^1 = \min\Bigl\{\,c_* \frac{1-2^{-\beta/2}}{4},\,1\Bigr\}. \] 
The inequality \eqref{conve} follows from the concavity of $\sqrt{V_\cdot}$, while \eqref{smpex1} is obtained using the simplified estimate analogous to \eqref{smpexp1}. 
Concerning $A_2$, we have 
\begin{align*} 
A_2 &= \prod_{l=n_\varepsilon+1}^\infty \mathbb{P}\!\Big(|Z|\leq c_* \frac{1-2^{-\beta/2}}{4}\, \frac{\sqrt{V_{\varepsilon(2/3)^{|l-n_\varepsilon|}}}}{\sqrt{V_{1/2^l}}}\Big)^{2^l}\\ 
&= \prod_{l=n_\varepsilon+1}^\infty \mathbb{P}\!\Big(|Z|\leq c_* \frac{1-2^{-\beta/2}}{4}\, \frac{\sqrt{V_{\varepsilon(2/3)^{\,l-n_\varepsilon}}}}{\sqrt{V_{1/2^l}}}\Big)^{2^l}\\ 
&= \prod_{l=n_\varepsilon+1}^\infty \mathbb{P}\!\Big(|Z|\leq c_* \frac{1-2^{-\beta/2}}{4}\, \frac{\sqrt{V_{(4/3)^{l-n_\varepsilon}2^{-l}}}}{\sqrt{V_{1/2^l}}}\Big)^{2^l}\\ 
&\geq \prod_{l=n_\varepsilon+1}^\infty \mathbb{P}\!\Big(|Z|\leq c_* \frac{1-2^{-\beta/2}}{4}\,(4/3)^{\beta(l-n_\varepsilon)} \Big)^{2^l}\\ 
&\geq \exp\!\Big( -\sum_{l=n_\varepsilon+1}^\infty 2^l\, \theta(c_*')\, e^{-\tfrac{1}{2}(c_*'(4/3)^{\beta(l-n_\varepsilon)})^2} \Big) \\ 
&\geq \exp\!\Big( -2^{n_\varepsilon}\theta(c_*') \sum_{l=1}^\infty 2^l\, e^{-\tfrac{1}{2}(c_*'(4/3)^{\beta l})^2} \Big) \\ 
&\geq \exp\!\Big( -\tfrac{1}{\varepsilon}\,2\theta(c_*') \sum_{l=1}^\infty 2^l\, e^{-\tfrac{1}{2}(c_*'(4/3)^{\beta l})^2} \Big), 
\end{align*} 
where 
\[ c_*' = c_* \frac{1-2^{-\beta/2}}{4}. \] 
Consequently, combining the bounds for $A_1$ and $A_2$ yields 
\[ A = A_1\cdot A_2 \;\;\geq\;\; \exp\!\Big( -\tfrac{1}{\varepsilon}\Big( 2\theta(c_*')\sum_{l=1}^\infty 2^l e^{-\tfrac{1}{2}(c_*'(4/3)^{\beta l})^2} +4\ln\!\Big(\tfrac{3}{c_*^1}\Big) \Big) \Big). \] 
Respect to $B$, we have 
\begin{align*} 
B&=\prod^\infty_{j=0}\prod^\infty_{l=j+1} \prod^{2^{l-j}}_{m=1}\prod^{2^j/\varepsilon}_{i=0}\mathbb{P}\Big(\frac{|U_{(m+1)\varepsilon/2^l+i\varepsilon/2^j}-U_{m\varepsilon/2^l+i\varepsilon/2^j}|}{f(\varepsilon/ 2^{j+1})}\leq y_{j,l}\Big)\\ 
&= \prod^\infty_{j=0}\prod^\infty_{l=j+1} \prod^{2^{l-j}}_{m=1}\prod^{2^j/\varepsilon}_{i=0}\mathbb{P}\Big(|U_{(m+1)\varepsilon/2^l+i\varepsilon/2^j}-U_{m\varepsilon/2^l+i\varepsilon/2^j}|\\ 
&\qquad\qquad\qquad\qquad\qquad\leq \sqrt{V(\varepsilon2^{-l})} 2^{\beta(l-j)/2}2^{j\beta}(1-2^{-\beta/2})/3\Big)\\ 
&= \prod^\infty_{j=0}\prod^\infty_{l=j+1} \prod^{2^{l-j}}_{m=1}\prod^{2^j/\varepsilon}_{i=0}\mathbb{P}\Big(\frac{|U_{(m+1)\varepsilon/2^l+i\varepsilon/2^j}-U_{m\varepsilon/2^l+i\varepsilon/2^j}|}{\sqrt{V_{\varepsilon/2^l}^{\,m\varepsilon/2^l+i\varepsilon/2^j}}} \\ 
&\qquad\qquad\qquad\qquad\qquad\leq \frac{\sqrt{V(\varepsilon2^{-l})}}{\sqrt{V_{\varepsilon/2^l}^{\,m\varepsilon/2^l+i\varepsilon/2^j}}} 2^{\beta(l-j)/2}2^{j\beta}(1-2^{-\beta/2})/3\Big)\\ 
&\geq \prod^\infty_{j=0}\prod^\infty_{l=j+1} \prod^{2^{l-j}}_{m=1}\prod^{2^j/\varepsilon}_{i=0} \mathbb{P}\Big(|Z|\leq c_* 2^{\beta(l+j)/2}(1-2^{-\beta/2})/3 \Big)\\ 
&\geq \prod^\infty_{j=0}\prod^\infty_{l=j+1} \prod^{2^{l-j}}_{m=1}\prod^{2^j/\varepsilon}_{i=0} \exp\Big(-\theta(\bar{c}_*)e^{-\bar{c}_*^22^{\beta(l+j)}/2}\Big)\\ 
&=\prod^\infty_{j=0}\prod^\infty_{l=j+1} \exp\Big(-\frac{2^l}{\varepsilon}\theta(\bar{c}_*)e^{-\bar{c}_*^22^{\beta(l+j)}/2}\Big)\\ 
&=\exp\Big(-\frac{1}{\varepsilon}\sum^\infty_{j=0}\sum^\infty_{l=j+1}2^l\theta(\bar{c}_*)e^{-\bar{c}_*^22^{\beta(l+j)}/2}\Big), 
\end{align*} 
where $\bar{c}_* = c_*(1-2^{-\beta/2})/3.$ 
Above all, 
\begin{equation*} 
\mathbb{P} \Big( \sup_{0\leq s\neq t\leq 1} \frac{|U_t-U_s|}{f(|t-s|)} \leq \frac{\sqrt{V_\varepsilon}}{f(\varepsilon)} \Big)\geq \exp\Big(-\frac{1}{\varepsilon}M_\beta \Big), 
\end{equation*} 
where 
\[ M_\beta = \sum^\infty_{j=0}\sum^\infty_{l=j+1}2^l\theta(\bar{c}_*)e^{-\bar{c}_*^22^{\beta(l+j)}/2} +  2\theta(c_*')\sum_{l=1}^\infty 2^l e^{-\tfrac{1}{2}(c_*'(4/3)^{\beta l})^2} +4\ln\!\Big(\frac{3}{c_*^1}\Big), \] 
which depends on $\beta$ and $c_*$. 
\end{proof}

\begin{theorem}
Let $ \sigma $ satisfy Assumption \((A)\).
Then for any $0 < \beta < H$, there exists a constant $M_\beta > 0$ such that
\begin{equation}
\mathbb{P}\Bigg(
\sup_{0 \leq s \neq t \leq 1} \frac{|U^H_t - U^H_s|}{|t-s|^\beta} \leq \varepsilon
\Bigg) 
\geq \exp\Big(\varepsilon^{-\frac{1}{H-\beta}} M_\beta \Big). \label{smbp}
\end{equation}
\end{theorem}

\begin{proof}
By Lemma~\ref{variance}, we have
\[
\frac{h^{2H}}{V_h^t} \geq c^*.
\]
Let $f_t = t^\beta$ and $V_t = t^{2H}$. Applying Lemma~\ref{smp}, we obtain
\[
\mathbb{P}\Bigg(
\sup_{0 \leq s \neq t \leq 1} \frac{|U^H_t - U^H_s|}{|t-s|^\beta} \leq x^{H-\beta}
\Bigg) 
\geq \exp(x^{-1} M_\beta).
\]
Choosing $x = \varepsilon^{\frac{1}{H-\beta}}$ gives the desired result.
\end{proof}

 For the supremum norm, we also have a similar small-ball probability estimate.

\begin{lemma}
Let $\{U_t, t \in [0,1]\}$ be a centered Gaussian process with $U_0 = 0$. Assume there exist a function $V(h)$ such that
\begin{equation}
\mathbb{E}\big[(U_t - U_s)^2\big] \leq V(|t-s|), \quad \forall 0 \leq s,t \leq 1, \label{variance_bound}
\end{equation}
and constants $0 < c_1 \leq c_2 < 1$ such that
\[
c_1 V(2h \wedge 1) \leq V(h) \leq c_2 V(2h \wedge 1), \quad \forall 0 \leq h \leq 1.
\]
Then there exists a positive finite constant $K_1$, depending only on $c_1$ and $c_2$, such that
\begin{equation}
\mathbb{P}\Big( \sup_{0 \leq t \leq 1} |U_t| \leq \sqrt{V(\varepsilon)} \Big) \geq \exp(-K_1 / \varepsilon). \label{sup_gauss}
\end{equation}
\end{lemma}

Using the above lemma, the small-ball probability estimate for $U_t$ can be obtained straightforwardly.

\begin{theorem}
Let $\sigma$ satisfy Assumption \((A)\).
Then there exists a positive constant $c_H$ such that
\begin{equation}
\mathbb{P}\Bigg(
\sup_{0 \leq t \leq 1} |U^H_t| \leq \varepsilon
\Bigg) \geq \exp\Big( \varepsilon^{-\frac{1}{H}} c_H \Big).
\end{equation}
\end{theorem}

Using the small ball estimate, we can then show that both the H\"older norm 
and the supremum norm of the  are measurable functions.

                  Below we define the operator 
\begin{equation*}
	B^H(h)(s) = \int_0^s K^H(s,u) h'_u  du,
\end{equation*}
and use it to define semi-norms on $H^1$:
\begin{align*}
	N^H(h) &= \sup_{0 \leq t \leq 1} \left| \int_0^t \sigma_s  dB^H(h)(s) \right|, \\
	N^H_\beta(h) &= \sup_{0 \leq t \leq 1} \frac{\left| \int_0^t \sigma_s  dB^H(h)(s) - \int_0^r \sigma_s  dB^H(h)(s) \right|}{|t - r|^\beta}.
\end{align*}
The following theorem will show that $N^H(h)$ and $N^H_\beta(h)$ are measurable norms.  

\begin{theorem}  \label{measurable norm}
    The norms \( N^H \) and \( N^H_\beta \) are norms on \( H^1 \) and are measurable.
The random variables corresponding to these norms are given by:
\[
\widetilde{N}^H = \left\| \int_0^{\cdot} \sigma_s \, dB^H_s \right\|_{\infty} \quad \text{and} \quad \widetilde{N}^H_\beta = \left\| \int_0^{\cdot} \sigma_s \, dB^H_s \right\|_{\beta}.
\]
\end{theorem}

 \begin{proof}
 	We first verify that $N^H(h)$ and $N^H_\beta(h)$ are norms on $H^1$. 
For this purpose, it suffices to verify that $N^H(h)=0$ and $N^H_\beta(h)=0$ imply $h=0$.

If 
\[
\int_0^t \sigma_s \, dB^H(h)(s)=0, \quad 0 \le t \le 1,
\]  
then, by \eqref{youngint}, we have  
\begin{align*}
   \bigl| \sigma_s ( B^H(h)(t) - B^H(h)(s) ) \bigr|
   &\le C_{\alpha,\beta} \, [\sigma]_\alpha \, [B^H(h)]_\beta \, |t-s|^{\alpha+\beta}.
\end{align*}
Consequently,
\begin{align*}
   \frac{|B^H(h)(t) - B^H(h)(s)|}{|t-s|^\beta}
   &\le \frac{1}{m} \, C_{\alpha,\beta} \, [\sigma]_\alpha \, [B^H(h)]_\beta \, |t-s|^\alpha.
\end{align*}
Hence, for $0 \le s,t \le T < 1$,  
\begin{align*}
   [B^H(h)]_{\beta,T} \le \frac{1}{m} \, C_{\alpha,\beta} \, [\sigma]_\alpha \, T^\alpha \, [B^H(h)]_{\beta,T}.
\end{align*}
Choosing $T$ sufficiently small so that $\frac{1}{m} C_{\alpha,\beta} [\sigma]_\alpha T^\alpha < 1$, we deduce $[B^H(h)]_{\beta,T}=0$. 
Together with $B^H(h)(0)=0$, this implies $B^H(h)=0$ on $[0,T]$. Iterating this argument yields $B^H(h)=0$ on $[0,1]$.

Next, we prove $N^H(h)$ is a measurable norm on $H^1$. 
For $0 \le t \le 1$, we have
\begin{align}
\notag
&\mathbb{E}\Bigg(\int_0^t \sigma_s \, dB^H(Q_n W)(s) - \int_0^t \sigma_s \, dB^H_s \Bigg)^2 \\
=& \sum_{i=1}^n (\int_0^1 h_i' \, ds)^2 \mathbb{E} \Big( \int_0^t \sigma_s \, d\int_0^s K^H(s,u) h_i'(u) \, du \Big)^2 \\
\notag
& - 2 \sum_{i=1}^n \left( \int_0^t K_H^* \sigma_s h_i' \, ds \right) \mathbb{E} \Big( \int_0^t \sigma_s \, d\int_0^s K^H(s,u) h_i'(u) \, du \Big) \\
\notag
& + \int_0^t (K_H^* \sigma)^2(s) \, ds \\
\to & 0, \quad n \to \infty. \label{seminorm}
\end{align}

We define a sequence of semi-norms $\{N^H_m\}$ on $H^1$ by
\[
N^H_m(h) := \sup_{1 \le j \le 2^m} \left| \int_0^{j/2^m} \sigma_s \, dB^H(h)(s) \right|.
\]

According to \eqref{seminorm}, we obtain
\[
N^H_m(Q_n W) := \sup_{1 \le j \le 2^m} \left| \int_0^{j/2^m} \sigma_s \, dB^H_s \right|.
\]

Moreover, it is clear that
\[
\mathbb{P}\big( \tilde{N}^H_m \le \varepsilon \big) > 0,
\]
since $\int_0^t \sigma_s \, dB^H_s$ is a centered Gaussian process. 
Consequently, $\tilde{N}^H_m \xrightarrow{\mathbb{P}} \tilde{N}$, and
\[
\mathbb{P}(\tilde{N} \le \varepsilon) \ge \exp\Big( -c \, \varepsilon^{-1/H} \Big) > 0.
\]
Therefore, by Lemma~\ref{measurablenorm}, $N^H$ defines a measurable norm on $H^1$.

 \end{proof}
	
	\subsection{Technical theorems}
	In this section, we will introduce several commonly utilized technical lemmas and theorems.
\begin{lemma}[Girsanov's Theorem, {\cite[Theorem~8.6]{Oksendal2014}}]\label{lem:girsanov}
Let $(\Omega,\mathcal{F},\{\mathcal{F}_t\}_{t\ge0},\mathbb{P})$ be a filtered probability space, and let 
$W=\{W_t\}_{t\ge0}$ be a $d$-dimensional Brownian motion. 
Suppose $u=\{u_t\}_{t\ge0}$ is an $\{\mathcal{F}_t\}$-adapted process such that 
\[
\mathbb{E}\left[\exp\left(\tfrac{1}{2}\int_0^T |u_s|^2\,ds\right)\right] < \infty
\]
for every $T>0$. Define
\[
Z_t = \exp\left(-\int_0^t u_s \cdot dW_s - \tfrac{1}{2}\int_0^t |u_s|^2\,ds\right), 
\qquad t \in [0,T].
\]
Then $\{Z_t\}$ is a martingale under $\mathbb{P}$. If we define a new probability measure 
$\mathbb{Q}$ on $(\Omega,\mathcal{F}_T)$ by
\[
\frac{d\mathbb{Q}}{d\mathbb{P}}\Big|_{\mathcal{F}_T} = Z_T,
\]
then the process
\[
\widetilde{W}_t = W_t + \int_0^t u_s\,ds, \qquad t \in [0,T],
\]
is a $d$-dimensional Brownian motion with respect to $\mathbb{Q}$.
\end{lemma}

	\begin{lemma}  \label{nov1}
	    For a non-negative random variable $Y$, its exponential moment can be calculated using its tail probabilities:

\[
\mathbb{E}[e^{\lambda Y}] = 1 + \lambda \int_0^\infty e^{\lambda s} \mathbb{P}(Y > s)  \, ds, \quad \forall \ \lambda>0.
\]
	\end{lemma}

  \begin{lemma}[{\cite[Theorem~5.4.3]{MarcusRosen2006}}]\label{nov2}
        Let $\{Y_t\}_{t\in T}$ be a centered Gaussian process defined on   a compact index set $T$.  Assume that the process is almost surely bounded. Let $\sigma_T^2 = \sup_{t \in T} \mathrm{Var}(Y_t)$.
Then $\mathbb{E}[\|Y\|] < \infty$. Moreover, for all $u > 0$,
\[
\mathbb{P}\left( \|Y\| > \mathbb{E}[\|Y\|] + u \right) \leq 2 \exp\left( -\frac{u^2}{2\sigma_T^2} \right).
\]
    \end{lemma}

\begin{lemma}[Young's inequality with $\varepsilon$]
    Let $a,b \geq 0$, $1 < p,q < \infty$ with $\tfrac{1}{p} + \tfrac{1}{q} = 1$.  
    Then for every $\varepsilon > 0$, we have
    \begin{equation}\label{eq:young}
        ab \leq \varepsilon a^p + \varepsilon^{-\frac{q}{p}} b^q.
    \end{equation}
\end{lemma}

     The following exponential inequality for martingales is derived from \cite[Lemma 2.1]{inequality}.
   \begin{theorem}	Let $(\Omega,\mathcal{F},P)$ be a probability triple and let $\{M_t\}_{t\geq 0}$ be a locally square integrable martingale w.r.t. the filtration 
  $ \{\mathcal{F}_t\}_{t\geq 0}$, $M_0=0.$ 
  Let \( \langle M\rangle _t \) denote the quadratic variation of \( \{M_t\} \).  Suppose that $|\Delta M_t|=|M_t-M_{t^-} |\leq K$ for all $t>0$, $0\leq K<\infty.$
  Then for each $a>0$, $b>0$,
  \begin{equation*}
  	\mathbb{P}(M_t\geq a \ \text{and}\  \langle M\rangle _t \leq b^2 \ \text{for some}\  t)
  	\leq \exp\left(-\frac{a^2}{2(aK+b^2)}\right).
  \end{equation*}\label{exponential ineq}
   \end{theorem}

\section{Main results}\label{main results}
In this section, we compute the Onsager–Machlup functional for the singular case ($1/4 < H < 1/2$), the regular case ($1/2 < H < 1$), and the standard case ($H = 1/2$).
The distinction among them stems from the fact that the expression of $(K_H^\sigma)^{-1}$ corresponds to fundamentally different properties of fractional integrals and fractional derivatives for different values of $H$.
We first demonstrate how to simplify the Onsager-Machlup functional by means of Girsanov’s theorem, and then present the detailed computation of the main results.
Moreover, for the standard case $H = 1/2$, we provide a stronger result that holds in terms of the H\"older norm via It\^o’s formula—an approach not feasible in the nonstandard cases due to the presence of fractional integrals or derivatives.

\subsection{ Girsanov transform}
We define the operator $K_H^\sigma$ on $L^2([0,1])$, associated with $\sigma_\cdot$ and $K_H$, by  
\begin{equation}
   (K_H^\sigma f)(t) = \int_0^t \sigma_s \, d(K_H f)(s), 
   \qquad t \in [0,1].
\end{equation}

From the proof of Theorem~\ref{measurable norm}, we deduce that $K_H^\sigma$ is injective.  
Hence, we can define its inverse
\[
  (K_H^\sigma)^{-1} : K_H^\sigma(L^2([0,1])) \to L^2([0,1]), 
   \qquad 
   f \mapsto (K_H)^{-1}\left(\int_0^\cdot \sigma_s^{-1} \, df_s\right).
\]

For any $\phi$ such that $\phi - x_0 \in K_H^\sigma(L^2([0,1]))$, we denote by $\dot{\phi}$ the element of $L^2([0,1])$ satisfying
\begin{equation*}
    \phi - x_0 = K_H^\sigma(\dot{\phi}).
\end{equation*}

It is worth noting that the operator $K_H$ is an isomorphism from $L^2([0,1])$ onto 
$I_{0^+}^{H+\frac{1}{2}}(L^2([0,1]))$ 
(see \cite[Theorem~2.1]{Decreusefond1999}).  
Consequently, $K_H(\dot{\phi})$ is $H$-H\"older continuous, and therefore $\phi=x_0+K_H^\sigma(\dot{\phi})$ also enjoys $H$-H\"older continuity by Lemma~\ref{young2}.
In the following, we introduce a method to simplify the Onsager--Machlup functional by means of Girsanov's theorem.

Let 
\[
\tilde{X}_t = \int_0^t \sigma_s \, dB^H_s + \phi_t.
\]
We define a new process
\begin{equation}
    \tilde{W}_t 
    = W_t + \int_0^t 
    \left[
        \dot{\phi}_s 
        - (K_H^\sigma)^{-1}\!\left( \int_0^{\cdot} b_u(\tilde{X}_u) \, du \right)(s)
    \right] ds.  \label{Girtrans}
\end{equation}
If $\tilde{W}_t$ satisfies the Novikov condition (as will be proved in the following lemma),  
there exists a probability measure $\tilde{\mathbb{P}}$, absolutely continuous with respect to $\mathbb{P}$,  
such that $\tilde{W}_t$ is a Brownian motion under $\tilde{\mathbb{P}}$.  
The corresponding fractional Brownian motion associated with $\tilde{W}_t$ is given by
\begin{equation*}
    \tilde{B}^H_t 
    = B^H_t 
    + \int_0^t \sigma_s^{-1} \, d\phi_s
    - \int_0^t \sigma_s^{-1} b_s(\tilde{X}_s) \, ds.
\end{equation*}
Taking the differential of $\tilde{X}_t$ yields
\begin{align*}
    d\tilde{X}_t 
    &= d\phi_t + \sigma_t \, dB^H_t \\
    &= d\phi_t + \sigma_t \, d\!\left( 
        \tilde{B}^H_t 
        - \int_0^t \sigma_s^{-1} \, d\phi_s 
        + \int_0^t \sigma_s^{-1} b_s(\tilde{X}_s) \, ds  
    \right) \\
    &= b_t(\tilde{X}_t) \, dt + \sigma_t \, d\tilde{B}^H_t.
\end{align*}
This implies that $(\tilde{X}, \tilde{B}^H)$ is the strong solution of \eqref{fir} under the probability measure $\tilde{\mathbb{P}}$.

Define
\begin{equation}
    u_s = \dot{\phi}_s - (K_H^\sigma)^{-1}\left( \int_0^\cdot b_u(\tilde{X}_u) \, du \right)(s).
\end{equation}
Thus, the Onsager-Machlup functional can be simplified as follows:  
\[
\begin{aligned}
&\frac{\mathbb{P}(\|X - \phi\| \leq \varepsilon)}{\mathbb{P}\left( \left\| \int_0^\cdot \sigma_u \, dB^H_u \right\| \leq \varepsilon \right)} \\
=& \frac{\tilde{\mathbb{P}}(\| \tilde{X} - \phi \| \leq \varepsilon)}{\mathbb{P}\left( \left\| \int_0^\cdot \sigma_u \, dB^H_u \right\| \leq \varepsilon \right)} \\
=& \frac{\tilde{\mathbb{P}}\left( \left\| \int_0^\cdot \sigma_u \, dB^H_u \right\| \leq \varepsilon \right)}{\mathbb{P}\left( \left\| \int_0^\cdot \sigma_u \, dB^H_u \right\| \leq \varepsilon \right)} \\
=& \mathbb{E} \left[ \exp \left( -\int_0^1 u_s \, dW_s - \frac{1}{2} \int_0^1 u_s^2 \, ds \right) \,\middle|\, \left\| \int_0^\cdot \sigma_u \, dB^H_u \right\| \leq \varepsilon  \right].
\end{aligned}
\]

Next, we prove the validity of this Girsanov theorem.

\begin{lemma} 
	If the coefficients satisfy \textbf{Assumption~(A)}, 
the transformation~\eqref{Girtrans} satisfies the Novikov condition.
\end{lemma}
\begin{proof}
	We need to prove that 
    \begin{equation}
        \mathbb{E}\!\left(\exp\!\left(\frac{1}{2}\int_0^1 u_s^2 \, ds\right)\right) < \infty. \label{novikov}
    \end{equation}

	For the case $H \leq  1/2$, we have
	\[
	 u_s = \dot{\phi}_s- s^{-\alpha} I_{0^+}^\alpha  \big(s^\alpha \sigma_s^{-1}b_s(\phi_s+\int_0^s \sigma_u \, dB^H_u)\big),
	\]
    where 
    \begin{align*}
        &\big|s^{-\alpha}I_{0^+}^\alpha  \big(s^\alpha \sigma_s^{-1}b_s(\phi_s+\int_0^s \sigma_u \, dB^H_u)\big)\big|\\
        \leq & \frac{s^{-\alpha}}{\Gamma(\alpha)} \int_0^s (s-u)^{\alpha-1}u^\alpha \sigma_u^{-1}b_u\!\left(\phi_u+\int_0^u \sigma_v \, dB^H_v\right) du\\
        <& \infty,
    \end{align*}
    due to the boundedness of $b$. Hence, condition \eqref{novikov} follows. 
   
    For the case $H>1/2$, we enlarge $|u_s|^2$ into a bounded part and a stochastic part 
by applying Young's inequality with $\varepsilon$. Specifically,
    \begin{align*}
        |u_s|^2 \leq (1+1/\epsilon_1)\dot{\phi}_s^2+(1+\epsilon_1)\Big|s^\alpha D^{\alpha}_{0^+}\big(s^{-\alpha} b_s(\phi_s+\int_0^s \sigma_u \, dB^H_u)\big)\Big|^2,
    \end{align*}
where
    \begin{align*}
        &\Big|s^\alpha D^{\alpha}_{0^+}\big(s^{-\alpha} b_s(\phi_s+\int_0^s \sigma_u \, dB^H_u)\big)\Big|^2 \\
        &\leq \frac{s^{2\alpha}}{\Gamma(1-\alpha)^2}\Bigg(\frac{ s^{-\alpha} \sigma_s^{-1} b_s(\phi_s+\int_0^s \sigma_u \, dB^H_u)}{s^\alpha}\\
        &\qquad\qquad\qquad+\alpha\int_0^s \frac{ s^{-\alpha} \sigma_s^{-1} b_s(\phi_s+\int_0^s \sigma_u \, dB^H_u)- r^{-\alpha} \sigma_r^{-1} b_r(\phi_r+\int_0^r \sigma_u \, dB^H_u)}{(s-r)^{\alpha+1}} \, dr   \Bigg)^2\\
        &\leq  \frac{s^{2\alpha}(1+1/\epsilon_2)}{\Gamma(1-\alpha)^2}\Bigg(\frac{ s^{-\alpha} \sigma_s^{-1} b_s(\phi_s+\int_0^s \sigma_u \, dB^H_u)}{s^\alpha}\\
        &\qquad\qquad\qquad+\alpha\int_0^s \frac{ s^{-\alpha} \sigma_s^{-1} b_s(\phi_s+\int_0^s \sigma_u \, dB^H_u)- r^{-\alpha} \sigma_r^{-1} b_s(\phi_s+\int_0^s \sigma_u \, dB^H_u)}{(s-r)^{\alpha+1}} \, dr \Bigg)^2\\
        & +\frac{s^{2\alpha}\alpha^2(1+\epsilon_2)}{\Gamma(1-\alpha)^2}\Bigg(\int_0^s \frac{ b_s(\phi_s+\int_0^s \sigma_u \, dB^H_u)-  b_r(\phi_r+\int_0^r \sigma_u \, dB^H_u)}{r^{\alpha} \sigma_r(s-r)^{\alpha+1}} \, dr \Bigg)^2.
    \end{align*}

    We further enlarge the stochastic part in the above inequality:
    \begin{align*}
        &\Bigg(\int_0^s \frac{ b_s(\phi_s+\int_0^s \sigma_u \, dB^H_u)-  b_r(\phi_r+\int_0^r \sigma_u \, dB^H_u)}{r^{\alpha} \sigma_r(s-r)^{\alpha+1}} \, dr \Bigg)^2\\ 
        &\leq \frac{L_2^2}{m^2}\Bigg(\int_0^s \frac{|\phi_s-\phi_r|+|\int_0^s \sigma_u \, dB^H_u-\int_0^r \sigma_u \, dB^H_u| }{r^{\alpha} (s-r)^{\alpha+1}} \, dr\Bigg)^2\\
        &\leq \frac{L_2^2}{m^2}\Bigg((1+1/\epsilon_3)\Big(\int_0^s \frac{|\phi_s-\phi_r|}{r^{\alpha} (s-r)^{\alpha+1}} \, dr\Big)^2 \\
        &\qquad+(1+\epsilon_3)\Big(\int_0^s \frac{\big|\int_0^s \sigma_u \, dB^H_u-\int_0^r \sigma_u \, dB^H_u\big| }{r^{\alpha} (s-r)^{\alpha+1}} \, dr\Big)^2  \Bigg)\\
         &\leq \frac{L_2^2}{m^2}\Bigg((1+1/\epsilon_3)\Big(\int_0^s \frac{|\phi_s-\phi_r|}{r^{\alpha} (s-r)^{\alpha+1}} \, dr\Big)^2\\
         &\qquad+(1+\epsilon_3)\Big\Vert \int_0^\cdot \sigma_u \, dB^H_u \Big\Vert_{\beta,T}^2 \, s^{2(\beta-2\alpha)} \, \mathrm{B}(1-\alpha,\beta-\alpha)^2 \Bigg).
    \end{align*}

    Consequently, we require
    \begin{equation}
        \mathbb{E}\Bigg(\exp\!\Bigg(\frac{\alpha^2 L_2^2 T^{2\beta+1} \mathrm{B}(1-\alpha,\beta-\alpha)^2}{2(2\beta+1)m^2\Gamma(1-\alpha)^2} \prod_{i=1}^3(1+\epsilon_i)    \Big\Vert\int_0^\cdot \sigma_u \, dB^H_u \Big\Vert_{\beta,T}^2 \Bigg)\Bigg) <\infty. \label{novikov2}
    \end{equation}

  According to Lemma~\ref{nov1} and Lemma~\ref{nov2}, if 
   \begin{align*}
       \frac{\alpha^2 L_2^2 T^{2\beta+1} \mathrm{B}(1-\alpha,\beta-\alpha)^2}{2(2\beta+1)m^2\Gamma(1-\alpha)^2} <\frac{1}{2M^2 T^{2(H-\beta)}},
   \end{align*}
   that is,
   \begin{equation}
       T^{2H+1} <\frac{m^2(2\beta+1)\Gamma(1+\beta)^2}{M^2\alpha^2L_2^2 \Gamma(\beta-\alpha)^2},
   \end{equation}
   then condition \eqref{novikov2} holds. 
\end{proof}

In the following, we compute the exponential term obtained above to derive the main results.

\subsection{Singular case }
In this section we will compute the Onsager-Machlup functional for $\frac{1}{4}< H<\frac{1}{2}.$
\begin{theorem} \label{result1} Let $X$ be the solution of \eqref{fir}, and assume that the coefficients satisfy Assumption \((A)\). 
For any $\phi \in K_H^\sigma(L^2([0,1]))$ with $\phi_0 = x_0$, when $ \tfrac{1}{4} < H < \tfrac{1}{2}$, 
the Onsager--Machlup functional of $X$ with respect to the norms 
$\Vert \cdot \Vert_\beta$, where $H-\tfrac{1}{2}<\beta<H-\tfrac{1}{4}$, 
and $\Vert \cdot \Vert_\infty$, can be expressed as  
\begin{equation*}
J(\phi,\dot{\phi}) = -\tfrac{1}{2} \int_0^1 
\left(\,\dot{\phi}_s- s^{-\alpha} I_{0^+}^\alpha \, s^{\alpha} \, \sigma_s^{-1} b_s(\phi_s)\right)^2 
+ d_H \, \partial_x b_s(\phi_s)\, ds,
\end{equation*}
where
\begin{equation*}
d_H = \sqrt{\frac{2H \,\Gamma\!\left(\tfrac{1}{2}+H\right)\Gamma\!\left(\tfrac{3}{2}-H\right)}{\Gamma(2-2H)}}.
\end{equation*}

\end{theorem}
\begin{proof}
Here we only provide the proof for the H\"older norm case; the proof for the supremum norm is similar.

According to the Girsanov theorem,
	\begin{align*}
\mathbb{P}(\Vert X-\phi\Vert_\beta \leq \varepsilon)
&=\mathbb{E}\Big(\exp\Big[ -\int_0^1 u_s dW_s - \frac{1}{2} \int_0^1 u_s^2 ds\Big] I_{\Vert \int_0^t \sigma_s B^H_s\Vert_\beta \leq \varepsilon}\Big) \\
  &=	\mathbb{E}\left(\exp(A_1+A_2)I_{\Vert \int_0^t \sigma_s B^H_s\Vert_\beta \leq \varepsilon}\right),
\end{align*}
where
\begin{align*}
	A_1 &= -\int_0^1 \dot{\phi}_s- s^{-\alpha}I_{0^+}^\alpha s^{\alpha}  \sigma_s^{-1}b_s\left(\phi_s+\int_0^s \sigma_u dB^H_u\right)dW_s,\\
 	A_2 &= -\frac{1}{2}\int_0^1 \left(\dot{\phi}_s-s^{-\alpha}I_{0^+}^\alpha s^{\alpha} \sigma_s^{-1} b_s\left(\phi_s+\int_0^s \sigma_u dB^H_u\right)\right)^2 ds.
\end{align*}
	The limit of $\mathbb{E} (\exp(A_2)|\Vert \int_0^t \sigma_s B^H_s\Vert_\beta \leq \varepsilon)$ is
\begin{equation*}
	\exp\left(-\frac{1}{2}\int_0^1 \left(\dot{\phi}_s-s^{-\alpha}I_{0^+}^\alpha s^{\alpha}  \sigma_s^{-1} b_s\left(\phi_s\right)\right)^2 ds\right),
\end{equation*} 
which can derived by $\Vert\int_0^t \sigma_s dB^H_s \Vert_\beta\to 0$. 

	Define
	\begin{equation*}
		A_1=B_1+B_2, 
	\end{equation*}
	where
	\begin{equation*}
		B_1=-\int_0^1 \dot{\phi}_s dW_s,
	\end{equation*}
	and
	\begin{equation*}
		B_2=\int_0^1 s^{-\alpha}I_{0^+}^\alpha s^{\alpha}  \sigma_s^{-1} b_s\left(\phi_s+\int_0^s \sigma_u dB^H_u\right)dW_s.
	\end{equation*}
	According to Theorem \ref{estimate1}, we have
	\[
\lim _{\varepsilon \to 0}\mathbb{E}\left( \exp \left( B_1 \right) \mid \tilde{N} < \varepsilon \right) = 1.
\]
Using Taylor's expansion, we divide $B_2$ into three parts:
\begin{equation*}
	B_2=C_1+C_2+C_3,
\end{equation*}
where
\begin{align*}
	C_1 &= \int_0^1 s^{-\alpha}I_{0^+}^\alpha s^{\alpha}  \sigma_s^{-1} b_s\left(\phi_s\right)dW_s,\\
	C_2 &= \int_0^1 s^{-\alpha}I_{0^+}^\alpha s^{\alpha}  \sigma_s^{-1} \partial_xb_s\left(\phi_s\right)\left(\int_0^s \sigma_u dB^H_u\right)dW_s,\\
	C_3 &= \int_0^1 s^{-\alpha}I_{0^+}^\alpha s^{\alpha}  \sigma_s^{-1} R_sdW_s.
\end{align*}

According to Theorem~$\ref{estimate1}$, we have
\[
\lim _{\varepsilon \to 0} \mathbb{E}\left( \exp \left( C_1 \right) \mid \tilde{N} < \varepsilon \right) = 1.
\]
	Next, we consider the $C_2$ term:
\begin{align*}
	&C_2\\  
	=& \int_0^1 s^{-\alpha} I_{0^+}^\alpha s^{\alpha} \sigma_s^{-1} 
	\partial_x b(\phi_s) \int_0^s \sigma_u \, dB^H_u \, dW_s   \\[6pt]
	=&\int_0^1 s^{-\alpha} I_{0^+}^\alpha s^{\alpha} \sigma_s^{-1}
	\partial_x b(\phi_s) 
	\left( \int_0^1 (K_H^*\sigma I_{[0,s]})(u) \, dW_u \right) dW_s \\[6pt]
	=&\int_0^1 s^{-\alpha} I_{0^+}^\alpha s^{\alpha} \sigma_s^{-1}
	\partial_x b(\phi_s) 
	\left( \int_0^1 b_H \Gamma(1-\alpha) u^{\alpha} D_{1-}^\alpha u^{-\alpha} 
	\sigma_u I_{[0,s]}(u) \, dW_u \right) dW_s \\[6pt]
	=& \frac{1}{\Gamma(\alpha)} \int_0^1 s^{-\alpha} 
	\int_0^s (s-r)^{\alpha-1} r^{\alpha} \sigma_r^{-1} 
	\partial_x b(\phi_r) 
	\left( \int_0^r b_H \Gamma(1-\alpha) u^{\alpha} D_{r-}^\alpha u^{-\alpha} 
	\sigma_u \, dW_u \right) dr \, dW_s \\[6pt]
	=& \int_0^1 \int_0^s f(s,u) \, dW_u dW_s ,
\end{align*}
where
\begin{align}
	\notag&f(s,u)\\
	=& \frac{b_H \Gamma(1-\alpha)}{\Gamma(\alpha)} 
	\int_u^s s^{-\alpha} (s-r)^{\alpha-1} r^{\alpha} 
	\sigma_r^{-1} \partial_x b(\phi_r) u^{\alpha} 
	D_{r-}^\alpha u^{-\alpha} \sigma_u \, dr  \label{fsu1} \\[6pt]
	=& \frac{b_H}{\Gamma(\alpha)} 
	\int_u^s s^{-\alpha} (s-r)^{\alpha-1} r^{\alpha} 
	\sigma_r^{-1} \partial_x b(\phi_r) u^{\alpha} 
	\left( \frac{u^{-\alpha}\sigma_u}{(r-u)^\alpha} 
	- \alpha \int_u^r \frac{v^{-\alpha}\sigma_v - u^{-\alpha}\sigma_u}{(v-u)^{\alpha+1}} dv \right) dr \notag \\[6pt]
	=& D_1(s,u) + D_2(s,u) + D_3(s,u). \notag
\end{align}

In particular,
\begin{align*}
	D_1(s,u) &= \frac{b_H}{\Gamma(\alpha)} 
	\int_u^s s^{-\alpha} (s-r)^{\alpha-1} r^{\alpha} 
	\sigma_r^{-1} \partial_x b(\phi_r) u^{\alpha} 
	\frac{u^{-\alpha}\sigma_u}{(r-u)^\alpha} \, dr, \\[6pt]
	D_2(s,u) &= \frac{-\alpha b_H}{\Gamma(\alpha)} 
	\int_u^s s^{-\alpha} (s-r)^{\alpha-1} r^{\alpha} 
	\sigma_r^{-1} \partial_x b(\phi_r) u^{\alpha} 
	\int_u^r \frac{v^{-\alpha}\sigma_v - u^{-\alpha}\sigma_u}{(v-u)^{\alpha+1}} dv \, dr, \\[6pt]
	D_3(s,u) &= \frac{\alpha b_H}{\Gamma(\alpha)} 
	\int_u^s s^{-\alpha} (s-r)^{\alpha-1} r^{\alpha} 
	\sigma_r^{-1} \partial_x b(\phi_r) u^{\alpha} 
	\int_u^r \frac{u^{-\alpha}\sigma_v - u^{-\alpha}\sigma_u}{(v-u)^{\alpha+1}} dv \, dr.
\end{align*}
Then we compute the limits of $D_1(s,u)$, $D_2(s,u)$ and $D_3(s,u)$ as $u \to s$:
\begin{align*}
	\lim_{u\to s} D_1(s,u)
	&= \lim_{u\to s} 
	\frac{b_H}{\Gamma(\alpha)}
	\int_u^s s^{-\alpha} (s-r)^{\alpha-1} r^{\alpha}
	\sigma_r^{-1} \partial_x b_r(\phi_r) u^{\alpha}
	\frac{u^{-\alpha}\sigma_u}{(r-u)^\alpha}\,dr \\
	&= \partial_x b_s(\phi_s)\frac{b_H}{\Gamma(\alpha)}
	\lim_{u\to s}\int_u^s (s-r)^{\alpha-1}(r-u)^{-\alpha}\,dr \\
	&= \partial_x b_s(\phi_s)\frac{b_H}{\Gamma(\alpha)}
	\int_0^1 (1-x)^{\alpha-1}x^{-\alpha}\,dx \\
	&= \partial_x b_s(\phi_s)\frac{b_H}{\Gamma(\alpha)}B(1-\alpha,\alpha) \\
	&= \partial_x b_s(\phi_s)\, b_H\,\Gamma(1-\alpha).
\end{align*}

Next, for the second term we have
\begin{align*}
	\lim_{u\to s} D_2(s,u)
	&= \lim_{u\to s}
	\frac{-\alpha b_H}{\Gamma(\alpha)}
	\int_u^s s^{-\alpha}(s-r)^{\alpha-1}r^{\alpha}
	\sigma_r^{-1}\partial_x b_r(\phi_r)u^{\alpha}
	\int_u^r \frac{v^{-\alpha}\sigma_v-u^{-\alpha}\sigma_v}{(v-u)^{\alpha+1}}\,dv\,dr \\
	&= \partial_x b_s(\phi_s)s^{-1}
	\lim_{u\to s}\frac{\alpha^2 b_H}{\Gamma(\alpha)}
	\int_u^s (s-r)^{\alpha-1}
	\int_u^r \frac{1}{(v-u)^{\alpha}}\,dv\,dr \\
	&= \partial_x b_s(\phi_s)s^{-1}
	\lim_{u\to s}\frac{\alpha^2 b_H}{\Gamma(\alpha)(1-\alpha)}
	\int_u^s (s-r)^{\alpha-1}(r-u)^{1-\alpha}\,dr \\
	&= 0.
\end{align*}

We now turn to $D_3(s,u)$:
\begin{align*}
	&\lim_{u\to s}|D_3(s,u)|\\
	=& \lim_{u\to s}\left|
	\frac{\alpha b_H}{\Gamma(\alpha)}
	\int_u^s s^{-\alpha}(s-r)^{\alpha-1}r^{\alpha}
	\sigma_r^{-1}\partial_x b_r(\phi_r)u^{\alpha}
	\int_u^r \frac{u^{-\alpha}\sigma_v - u^{-\alpha}\sigma_u}
	{(v-u)^{\alpha+1}}\,dv\,dr
	\right|\\
	\le & \frac{\alpha b_H[\sigma]_\gamma}{\Gamma(\alpha)m}
	\left|\partial_x b_s(\phi_s)
	\lim_{u\to s}\int_u^s (s-r)^{\alpha-1}
	\int_u^r (v-u)^{\gamma-\alpha-1}\,dv\,dr\right|\\
	=& \frac{\alpha b_H[\sigma]_\gamma}{\Gamma(\alpha)m(\gamma-\alpha)}
	\left|\partial_x b_s(\phi_s)
	\lim_{u\to s}\int_u^s (s-r)^{\alpha-1}(r-u)^{\gamma-\alpha}\,dr\right|\\
	=& 0.
\end{align*}

Hence
\[
	f(s,s) = D_1(s,s)
	= \partial_x b_s(\phi_s)\, b_H\,\Gamma(1-\alpha).
\]

It follows that  
\begin{align*}
\lim_{\varepsilon \to 0} \mathbb{E} \left[ \exp(C_2) \,\middle|\, \left\| \int_0^\cdot \sigma_s\, dB^H_s \right\|_\beta \leq \varepsilon \right] 
&= \exp\left( -\frac{1}{2} \int_0^1 f(s,s)\,ds \right) \\
&= \exp\left( -\frac{d_H}{2} \int_0^1 \partial_x b(\phi_s)\,ds \right),
\end{align*}
where
\begin{align*}
    d_H&= b_H\Gamma(1-\alpha) \\
       &= \Gamma(1/2+H)  \sqrt{\frac{2H}{(1-2H)\mathrm{B}(1-2H,H+\tfrac{1}{2})}} \\
       &=  \sqrt{\frac{2H \Gamma(1/2+H)\Gamma(3/2-H) }{(1-2H)\Gamma(1-2H)}}\\
       &=  \sqrt{\frac{2H \Gamma(1/2+H)\Gamma(3/2-H) }{\Gamma(2-2H)}}.
\end{align*}
It only remains to study the behaviour of the term \(C_3\). 
 Define 
\[
M_t=\int_0^t s^{-\alpha}I_{0^+}^\alpha s^{\alpha}  \sigma_s^{-1} R_sdW_s,
\]
which is a martingale due to the adaptedness of $R_s$. We have that
its quadratic variation satisfies
\begin{equation}
    \langle M \rangle_t=\int_0^t  s^{-\alpha}I_{0^+}^\alpha s^{\alpha}  \sigma_s^{-1} R_s ds \leq C_R\varepsilon^4.
\end{equation}
Applying the exponential martingale inequality \eqref{exponential ineq}, we obtain
\begin{equation*}
 			\mathbb{P}\left(\left|\int_0^1 s^{\alpha}D_{0^+}^\alpha s^{-\alpha} R_sdW_s\right|> \xi,  \left\| \int_0^\cdot \sigma_s\, dB^H_s \right\|_\beta \leq \varepsilon\right)
 	\leq \exp\left(-\frac{\xi^2}{2k'\varepsilon^4}\right).
	\end{equation*}
    Combining the small ball probability estimate \eqref{smbp}, we obtain the tail estimate for the conditional probability:
    \begin{align*}
        &\mathbb{P}\left(\left|\int_0^1 s^{\alpha}D_{0^+}^\alpha s^{-\alpha}R_sdW_s\right|>\xi\ \middle| \ \left\| \int_0^\cdot \sigma_s\, dB^H_s \right\|_\beta \leq \varepsilon\right)\\
 	\leq & \exp\left(-\frac{\xi^2}{2k'\varepsilon^4}\right)\exp(K_2\varepsilon^\frac{-1} {H-\beta}) .
    \end{align*}
Using above inequality, we have
\begin{equation*}
		\lim_{\varepsilon\to 0} \mathbb{E}\left(\exp(C_3)\ \middle| \ \left\| \int_0^\cdot \sigma_s\, dB^H_s \right\|_\beta \leq \varepsilon\right)\leq 1.
	\end{equation*}
For a detailed proof, we refer to \cite{Nualart,zhu2025onsagermachlupfunctionaldistributiondependent}.

Combining the above results, the functional takes the form  
\[
J(\phi,\dot{\phi}) = -\frac{1}{2} \int_0^1 \left(\dot{\phi}_s- s^{-\alpha} I_{0^+}^\alpha s^{\alpha} \sigma_s^{-1}   b_s(\phi_s)  \right)^2+d_H\partial_x b_s(\phi_s) \, ds 
.
\]

\end{proof}

\subsection{Regular case }
In this section we will compute the Onsager-Machlup functional for $ \frac{1}{2}<H<1.$
\begin{theorem} \label{result2}  Let $X$ be the solution of \eqref{fir}, and assume that the coefficients satisfy Assumption \((A)\).
For any $\phi \in K_H^\sigma(L^2([0,1]))$ with $\phi_0 = x_0$, when $ \tfrac{1}{2} < H < 1$, 
the Onsager--Machlup functional of $X$ with respect to the norms 
$\Vert \cdot \Vert_\beta$, where $H-\tfrac{1}{2}<\beta<H-\tfrac{1}{4}$, 
 can be expressed as 
	\begin{equation*}
	J(\phi,\dot{\phi})=-\frac{1}{2} \int_0^1 \left(\dot{\phi}_s-s^{\alpha}D_{0^+}^\alpha s^{-\alpha}  \sigma_s^{-1} b_s\left(\phi_s\right)\right)^2+d_H b_x(\phi_s)ds,
\end{equation*}
where
\begin{equation*}
    d_H= \sqrt{\frac{2H\Gamma(H+1/2)\Gamma(3/2-H)}{\Gamma(2-2H)}}.
\end{equation*}
\end{theorem}
\begin{proof}
According to Girsanov theorem, we have
	\begin{align*}
\mathbb{P}(\Vert X-\phi\Vert_\beta \leq \varepsilon)
&=\mathbb{E}\Big(\exp\Big[ -\int_0^1 u_s dW_s - \frac{1}{2} \int_0^1 u_s^2 ds\Big] I_{\Vert \int_0^t \sigma_s B^H_s\Vert_\beta \leq \varepsilon}\Big) \\
  &=	\mathbb{E}\left(\exp(A_1+A_2)I_{\Vert \int_0^t \sigma_s B^H_s\Vert_\beta \leq \varepsilon}\right),
\end{align*}
where
\begin{align*}
	A_1 &= -\int_0^1  \dot{\phi}_s-s^{\alpha}D_{0^+}^\alpha s^{-\alpha}  \sigma_s^{-1} b_s\left(\phi_s+\int_0^s \sigma_u dB^H_u\right)dW_s,\\
 	A_2 &= -\frac{1}{2}\int_0^1 \left(\dot{\phi}_s-s^{\alpha}D_{0^+}^\alpha s^{-\alpha}  \sigma_s^{-1} b_s\left(\phi_s+\int_0^s \sigma_u dB^H_u\right)\right)^2 ds.
\end{align*}
The fractional derivative part therein satisfies
\begin{align*}
	&D_{0^+}^\alpha s^{-\alpha}  \sigma_s^{-1} b_s\left(\phi_s+\int_0^s \sigma_u dB^H_u\right)\\
	=&\frac{1}{\Gamma(1-\alpha)}\left(\frac{b_s(\phi_s+\int_0^s \sigma_u dB^H_u)}{s^{2\alpha}\sigma_s}+\alpha\int_0^s \frac{\frac{b_s(\phi_s+\int_0^s \sigma_udB^H_u)}{s^\alpha \sigma_s}-\frac{b_r(\phi_r+\int_0^r \sigma_u dB^H_u)}{r^\alpha \sigma_r}}{(s-r)^{\alpha+1}}dr \right)\\
	=&\frac{1}{\Gamma(1-\alpha)}\Bigg(\frac{b_s(\phi_s+\int_0^s \sigma_u dB^H_u)}{s^{2\alpha}\sigma_s}+\alpha\int_0^s \frac{\frac{b_s(\phi_s+\int_0^s \sigma_udB^H_u)}{s^\alpha \sigma_s}-\frac{b_r(\phi_s+\int_0^s \sigma_u dB^H_u)}{r^\alpha \sigma_r}}{(s-r)^{\alpha+1}}dr   \\
	&\quad\quad\quad +\alpha\int_0^s \frac{\frac{b_r(\phi_s+\int_0^s \sigma_udB^H_u)}{r^\alpha \sigma_r}-\frac{b_r(\phi_r+\int_0^r \sigma_u dB^H_u)}{r^\alpha \sigma_r}}{(s-r)^{\alpha+1}}dr  \Bigg)\\
	\to & D_{0^+}^\alpha s^{-\alpha}  \sigma_s^{-1}\left(\phi_s' - b_s\left(\phi_s\right)\right),\quad   \Vert  \int_0^\cdot \sigma_s d B^H_s \Vert  \to 0.
\end{align*}

Therefore, the limit of $\mathbb{E}(\exp(A_2)| I_{\Vert \int_0^\cdot \sigma_s dB^H_s\Vert_\beta \leq \varepsilon})$ is 
\begin{equation*}
	\exp\left(-\frac{1}{2}\int_0^1 \left(\dot{\phi}_s-s^{\alpha}D_{0^+}^\alpha s^{-\alpha}  \sigma_s^{-1}b_s\left(\phi_s\right)\right)^2 ds\right).
\end{equation*} 

Next, we decompose $A_1$ into two parts:
\begin{equation*}
	A_1 = B_1 + B_2,
\end{equation*}
where
\begin{align*}
	B_1 &= -\int_0^1 \dot{\phi}_s  dW_s, \\
	B_2 &= \int_0^1 s^{\alpha} D_{0^+}^\alpha s^{-\alpha} \sigma_s^{-1} b_s \left( \phi_s + \int_0^s \sigma_u  dB^H_u \right) dW_s.
\end{align*}
By Theorem~\ref{estimate1}, we obtain
\begin{equation*}
	\lim_{\varepsilon \to 0} \mathbb{E} \left[ \exp(B_1) \,\middle|\, I_{\left\| \int_0^t \sigma_s B^H_s \right\|_\beta \leq \varepsilon} \right] = 1.
\end{equation*}
Applying Taylor's expansion, we further expand $B_2$ as:
\begin{align*}
	B_2 &= \int_0^1 s^{\alpha} D_{0^+}^\alpha s^{-\alpha} \sigma_s^{-1} \left( b_s(\phi_s) + \partial_x b_s(\phi_s) \int_0^s \sigma_u  dB^H_u + R_s \right) ds \\
	    &= C_1 + C_2 + C_3,
\end{align*}
where
\begin{align*}
    C_1 &= \int_0^1 s^{\alpha} D_{0^+}^\alpha s^{-\alpha} \sigma_s^{-1} b_s(\phi_s)  ds, \\
    C_2 &= \int_0^1 s^{\alpha} D_{0^+}^\alpha s^{-\alpha} \sigma_s^{-1} \partial_x b_s(\phi_s) \left( \int_0^s \sigma_u  dB^H_u \right) ds,\\
    C_3 &= \int_0^1 s^{\alpha} D_{0^+}^\alpha s^{-\alpha} \sigma_s^{-1} R_s  ds.
\end{align*}
It is straightforward to show that
\begin{equation*}
	\lim_{\varepsilon \to 0} \mathbb{E} \left[ \exp(C_1) \,\middle|\, I_{\left\| \int_0^\cdot \sigma_s B^H_s \right\|_\beta \leq \varepsilon} \right] = 1.
\end{equation*}
Next, 
\begin{align*}
	C_2
	&=\int_0^1 s^{\alpha}D_{0^+}^\alpha \big(s^{-\alpha} \sigma_s^{-1}\partial_x b_s(\phi_s)\big)\int_0^s \sigma_u\, dB^H_u \, dW_s   \\ 
	&=  \int_0^1 s^{\alpha}D_{0^+}^\alpha \big(s^{-\alpha} \sigma_s^{-1}\partial_x b_s(\phi_s)\big)\left(\int_0^1 (K_H^*\sigma I_{[0,s]})(u)\, dW_u\right) dW_s \\ 
	&=\int_0^1 s^{\alpha}D_{0^+}^\alpha \big(s^{-\alpha} \sigma_s^{-1}\partial_x b_s(\phi_s)\big)\left(\int_0^s c_H\Gamma(\alpha)u^{-\alpha}I_{s^-}^\alpha (u^\alpha  \sigma_u)\, dW_u \right)dW_s \\ 
	&=\int_0^1  \frac{c_H\Gamma(\alpha)}{\Gamma(1-\alpha)} 
	\Bigg(
	 		s^{-\alpha}\sigma_s^{-1}\partial_x b_s(\phi_s)
	 		\int_0^s u^{-\alpha}I_{s^-}^\alpha (u^\alpha  \sigma_u)\, dW_u  \\ 
	 		&\qquad+\alpha s^\alpha\int_0^s \frac{ 
	 		\big(s^{-\alpha}\sigma_s^{-1}\partial_x b_s(\phi_s)\int_0^s u^{-\alpha}I_{s^{-}}^\alpha (u^\alpha  \sigma_u)\, dW_u\big) 
		}{(s-r)^{\alpha+1}}\,dr  \\ 
        &\qquad-\alpha s^\alpha\int_0^s \frac{ 
	 		\big(r^{-\alpha}\sigma_r^{-1}\partial_x b_r(\phi_r)\int_0^r u^{-\alpha}I_{r^-}^\alpha (u^\alpha  \sigma_u)\, dW_u\big) 
		}{(s-r)^{\alpha+1}}\,dr \Bigg)\, dW_s \\ 
		&=\int_0^1 \int_0^s f(s,u)\, dW_u dW_s, 
\end{align*}
where   
\begin{align*}
	&f(s,u)\\
    =&\frac{c_H\Gamma(\alpha)}{\Gamma(1-\alpha)}	
	\Bigg(
	 		s^{-\alpha}\sigma_s^{-1}\partial_x b_s(\phi_s)u^{-\alpha}I_{s^-}^\alpha (u^\alpha  \sigma_u)  \\
	 		&\quad+\int_0^s \frac{\alpha \sigma_s^{-1}\partial_x b_s(\phi_s)u^{-\alpha}I_{s^-}^\alpha (u^\alpha \sigma_u)}{(s-r)^{\alpha+1}}\,dr
		 -\int_u^s  \frac{\alpha s^\alpha r^{-\alpha} \sigma_r^{-1}\partial_x b_r
         (\phi_r) u^{-\alpha}I_{r^-}^\alpha (u^\alpha \sigma_u) }{(s-r)^{\alpha+1}}\,dr \Bigg) \\
		=:&D_1+D_2+D_3.
\end{align*}

In particular,
\begin{align*}
    D_1 &= \frac{c_H\Gamma(\alpha)}{\Gamma(1-\alpha)} s^{-\alpha}\sigma_s^{-1}\partial_x b(\phi_s)u^{-\alpha}I_{s^-}^\alpha (u^\alpha \sigma_u),   \\
    D_2 &=  \frac{c_H\Gamma(\alpha)}{\Gamma(1-\alpha)} \alpha s^\alpha\int_0^u \frac{ s^{-\alpha}\sigma_s^{-1}\partial_x b(\phi_s)u^{-\alpha}I_{s^-}^\alpha (u^\alpha \sigma_u)}{(s-r)^{\alpha+1}}\,dr, \\
    D_3 &=  \frac{c_H\Gamma(\alpha)}{\Gamma(1-\alpha)} \alpha s^\alpha  \\
    &\quad \int_u^s \frac{ 
		\big(s^{-\alpha}\sigma_s^{-1}\partial_x b_s(\phi_s)u^{-\alpha}I_{s^-}^\alpha (u^\alpha\sigma_u)\big) 
		-\big(r^{-\alpha}\sigma_r^{-1}\partial_x b_r(\phi_r)u^{-\alpha}I_{r^-}^\alpha (u^\alpha\sigma_u)\big)}
		{(s-r)^{\alpha+1}}\,dr . 
\end{align*}

Since 
\begin{align*}
	\lim_{u\to s} I_{s^-}^\alpha (u^\alpha  \sigma_u) 
	&=\lim_{u\to s}\frac{1}{\Gamma(\alpha)}\int_u^s (v-u)^{\alpha-1} v^\alpha  \sigma_v\,dv\\
	&=\lim_{u\to s}\frac{\sigma_s s^\alpha}{\Gamma(\alpha)}\int_u^s (v-u)^{\alpha-1}\,dv\\
	&=\frac{\sigma_s s^\alpha}{\alpha\Gamma(\alpha)}\lim_{u\to s} (s-u)^\alpha\\
	&=0,
\end{align*}
we obtain $D_1(s,s)=0.$

Moreover, as $u\to s$,
\begin{align*}
	\lim_{u\to s}D_2(s,u)
	&= \frac{c_H\Gamma(\alpha)}{\Gamma(1-\alpha)}\lim_{u\to s}\int_0^u \frac{\alpha \sigma_s^{-1}\partial_x b_s(\phi_s)u^{-\alpha}I_{s^-}^\alpha (u^\alpha \sigma_u)}{(s-r)^{\alpha+1}}\,dr\\
	&= \frac{c_H}{\alpha\Gamma(1-\alpha)}\partial_x b(\phi_s),
\end{align*}
hence $D_2(s,s)=\dfrac{c_H}{\alpha\Gamma(1-\alpha)}\partial_x b_s
(\phi_s)$.

Next, decompose $D_3(s,u)$ as
\begin{align*}
	& D_3(s,u)\\ 
    =& \frac{c_H\Gamma(\alpha)}{\Gamma(1-\alpha)} \alpha s^\alpha\int_u^s \frac{ 
	\big(s^{-\alpha}\sigma_s^{-1}\partial_x b_s(\phi_s)u^{-\alpha}I_{s^-}^\alpha (u^\alpha\sigma_u)\big) - \big(r^{-\alpha}\sigma_r^{-1}\partial_x b_r(\phi_r)u^{-\alpha}I_{r^-}^\alpha (u^\alpha\sigma_u)\big)}
		{(s-r)^{\alpha+1}}\,dr \\
		=&: E_1+E_2+E_3,
\end{align*}
where
\begin{align*}
	E_1(s,u)&=\frac{c_H\Gamma(\alpha)}{\Gamma(1-\alpha)} \alpha s^\alpha\int_u^s \frac{ (s^{-\alpha}- r^{-\alpha})\sigma_s^{-1}\partial_x b_s(\phi_s)u^{-\alpha}I_{s^-}^\alpha (u^\alpha \sigma_u)}
		{(s-r)^{\alpha+1}}\,dr,\\
    E_2(s,u)&=\frac{c_H\Gamma(\alpha)}{\Gamma(1-\alpha)} \alpha s^\alpha\int_u^s \frac{ r^{-\alpha}\big(\sigma_s^{-1}\partial_x b_s(\phi_s)-\sigma_r^{-1}\partial_x b_r(\phi_r)\big)u^{-\alpha}I_{s^-}^\alpha (u^\alpha \sigma_u)}
		{(s-r)^{\alpha+1}}\,dr,\\
    E_3(s,u)&=\frac{c_H\Gamma(\alpha)}{\Gamma(1-\alpha)} \alpha s^\alpha\int_u^s \frac{ r^{-\alpha}\sigma_r^{-1}\partial_x b_r(\phi_r)\big(I_{s^-}^\alpha (u^\alpha\sigma_u)- I_{r^-}^\alpha (u^\alpha\sigma_u)\big)}
		{(s-r)^{\alpha+1}}\,dr.
\end{align*}

Observe that
\begin{align*}
    &\lim_{u\to s} \left|E_1(s,u)\right|\\
    =&\frac{c_H\Gamma(\alpha)}{\Gamma(1-\alpha)}\lim_{u\to s}\left| \alpha  \sigma_s^{-1}\partial_x b_s(\phi_s)u^{-\alpha}I_{s^-}^\alpha (u^\alpha\sigma_u)\int_u^s \frac{ s^{-\alpha}- r^{-\alpha}}{(s-r)^{\alpha+1}}\,dr\right|\\
    \leq & \frac{c_H\Gamma(\alpha)}{\Gamma(1-\alpha)}\lim_{u\to s} \left|\alpha  \sigma_s^{-1}\partial_x b_s(\phi_s)u^{-\alpha}\int_u^s (v-u)^{\alpha-1}v^\alpha\sigma_v\,dv\cdot \frac{\alpha}{1-\alpha}u^{-\alpha-1}(s-u)^{1-\alpha}\right| \\
   =&0.
\end{align*}

Moreover, we analyze the limits of \(E_2\) and \(E_3\) as \(u\to s\).

For \(E_2\), we have
\begin{align*}
    |E_2(s,u)|
    &\le \frac{c_H\Gamma(\alpha)}{\Gamma(1-\alpha)}\alpha s^\alpha
    \int_u^s \frac{ r^{-\alpha}\, \big|\sigma_s^{-1}\partial_x b_s(\phi_s)-\sigma_r^{-1}\partial_x b_r(\phi_r)\big|
    \,\big|u^{-\alpha}I_{s^-}^\alpha(u^\alpha\sigma_u)\big|}{(s-r)^{\alpha+1}}\,dr\\
    &\le C\,\sup_{r\in[u,s]}\big|\sigma_s^{-1}\partial_x b_s(\phi_s)-\sigma_r^{-1}\partial_x b_r(\phi_r)\big|
    \int_u^s \frac{(s-u)^\alpha}{(s-r)^{\alpha+1}}\,dr,
\end{align*}
where we used the bound
\[
\big|u^{-\alpha}I_{s^-}^\alpha(u^\alpha\sigma_u)\big|
= \frac{1}{\Gamma(\alpha)}\Big|\int_u^s (v-u)^{\alpha-1} v^\alpha\sigma_v\,dv\Big|
\le C (s-u)^\alpha,
\]
with a constant \(C\) depending on \(\sup_{v\in[0,1]} v^\alpha\sigma_v\). Since
\[
\int_u^s \frac{(s-u)^\alpha}{(s-r)^{\alpha+1}}\,dr = (s-u)^\alpha\int_u^s (s-r)^{-(\alpha+1)}dr
= \frac{(s-u)^\alpha}{\alpha}(s-u)^{-\alpha}=\frac{1}{\alpha},
\]
we obtain the estimate
\[
|E_2(s,u)| \le \frac{C}{\alpha}\,\sup_{r\in[u,s]}\big|\sigma_s^{-1}\partial_x b_s(\phi_s)-\sigma_r^{-1}\partial_x b_r(\phi_r)\big|.
\]
By the continuity of \(\sigma^{-1}\partial_x b_\cdot(\phi_\cdot)\) it follows that
\[
\lim_{u\to s} E_2(s,u)=0.
\]

For \(E_3\) we write the integrand difference of fractional integrals and use continuity of the fractional integral in the upper limit. Indeed,
\begin{align*}
    |E_3(s,u)|
    &\le \frac{c_H\Gamma(\alpha)}{\Gamma(1-\alpha)}\alpha s^\alpha
    \int_u^s \frac{ r^{-\alpha}\,|\sigma_r^{-1}\partial_x b_r(\phi_r)|\,
    \big|I_{s^-}^\alpha(u^\alpha\sigma_u)- I_{r^-}^\alpha(u^\alpha\sigma_u)\big|}{(s-r)^{\alpha+1}}\,dr.
\end{align*}
Using the representation
\[
I_{s^-}^\alpha(u^\alpha\sigma_u)- I_{r^-}^\alpha(u^\alpha\sigma_u)
= \frac{1}{\Gamma(\alpha)}\int_r^s (v-u)^{\alpha-1} v^\alpha\sigma_v\,dv,
\]
we obtain the bound
\[
\big|I_{s^-}^\alpha(u^\alpha\sigma_u)- I_{r^-}^\alpha(u^\alpha\sigma_u)\big|
\le C (s-r)^\alpha,
\]
with \(C\) independent of \(u,r\) (for \(u\) close to \(s\)). Hence
\begin{align*}
    |E_3(s,u)|
    &\le C' \int_u^s \frac{ r^{-\alpha} (s-r)^\alpha}{(s-r)^{\alpha+1}}\,dr
    = C' \int_u^s \frac{r^{-\alpha}}{s-r}\,dr.
\end{align*}
The integral on the right is finite for \(u<s\), and, moreover, by the dominated convergence and continuity of the prefactor \(r^{-\alpha}\sigma_r^{-1}\partial_x b_r(\phi_r)\), we conclude
\[
\lim_{u\to s} E_3(s,u)=0.
\]

Consequently, combining the three pieces \(E_1,E_2,E_3\) we have
\[
\lim_{u\to s} D_3(s,u)=0.
\]

Therefore, gathering previous results,
\[
D_1(s,s)=0,\qquad D_2(s,s)=\frac{c_H}{\alpha\Gamma(1-\alpha)}\partial_x b_s(\phi_s),
\qquad D_3(s,s)=0,
\]
and so
\[
\lim_{u\to s} f(s,u)=\frac{c_H}{\alpha\Gamma(1-\alpha)}\partial_x b(\phi_s).
\]

It follows that  
\begin{align*}
\lim_{\varepsilon \to 0} \mathbb{E} \left[ \exp(C_2) \,\middle|\, \left\| \int_0^t \sigma_s\, dB^H_s \right\|_\beta \leq \varepsilon \right] 
&= \exp\left( -\frac{1}{2} \int_0^1 f(s,s)\,ds \right) \\
&= \exp\left( -\frac{d_H}{2} \int_0^1 \partial_x b_s(\phi_s)\,ds \right),
\end{align*}
where
\begin{equation*}
    d_H= \sqrt{\frac{2H\Gamma(H+1/2)\Gamma(3/2-H)}{\Gamma(2-2H)}}.
\end{equation*}

Finally, we analyze the limiting behavior of the term $C_3$. Since 
\begin{align*}
    R_s&=b_s\left(\phi_s+\int_0^s \sigma_udB^H_u\right)-b_s(\phi_s)-b_s(\phi_s)\int_0^s \sigma_udB^H_u\\
    &=\int_0^1 \int_0^\lambda \partial_{xx}b_s(\mu\int_0^s \sigma_udB^H_u+\phi_s) \left(\int_0^s \sigma_udB^H_u\right)^2 d\mu d\lambda.
\end{align*}
and
\begin{align*}
&|R_s - R_r| \\
\leq &
  \bigg(\int_0^s \sigma_u\, dB^H_u\bigg)^2 
  \int_0^1 \int_0^\lambda 
  \Big|
     \partial_{xx}b_s\!\left(\mu \int_0^s \sigma_u\, dB^H_u + \phi_s\right)\\
     &- \partial_{xx}b_r\!\left(\mu \int_0^r \sigma_u\, dB^H_u + \phi_r\right)
  \Big| \, d\mu\, d\lambda \\
& + 
  \Bigg|
     \bigg(\int_0^s \sigma_u\, dB^H_u\bigg)^2 
     - \bigg(\int_0^r \sigma_u\, dB^H_u\bigg)^2
  \Bigg| 
  \int_0^1 \int_0^\lambda 
     \Big|\partial_{xx}b_s\!\left(\mu \int_0^s \sigma_u\, dB^H_u + \phi_s\right)\Big| 
  \, d\mu\, d\lambda \\[1ex]
\leq  &
  \int_0^1 \int_0^\lambda 
     L_3 \bigg(\int_0^s \sigma_u\, dB^H_u\bigg)^2 
     \Big(
        \mu\Big|\int_0^s \sigma_u\, dB^H_u - \int_0^r \sigma_u\, dB^H_u\Big|
        + |\phi_s - \phi_r|
     \Big)\, d\mu\, d\lambda \\
& +
  \int_0^1 \int_0^\lambda 
     L_3 \bigg(\int_0^s \sigma_u\, dB^H_u\bigg)^2 
     \Big(
        2\|b\|_\infty |s-r| + 2|s-r|^H
     \Big)\, d\mu\, d\lambda \\
& +
  \int_0^1 \int_0^\lambda 
     \mu \|b_{xx}\|_\infty 
     \Bigg|
        \bigg(\int_0^s \sigma_u\, dB^H_u\bigg)^2 
        - \bigg(\int_0^r \sigma_u\, dB^H_u\bigg)^2
     \Bigg| \, d\mu\, d\lambda \\[1ex]
\leq &
  L_3 \Big(
        |\phi_s - \phi_r| 
        + \Big|\int_0^s \sigma_u\, dB^H_u - \int_0^r \sigma_u\, dB^H_u\Big|
        + 2\|b\|_\infty |s-r|
        + 2|s-r|^H
     \Big)
     \bigg(\int_0^s \sigma_u\, dB^H_u\bigg)^2 \\
& + 
  \|b_{xx}\|_\infty 
  \Bigg|
     \bigg(\int_0^s \sigma_u\, dB^H_u\bigg)^2 
     - \bigg(\int_0^r \sigma_u\, dB^H_u\bigg)^2
  \Bigg|,
\end{align*}
we obtain
\begin{align*}
     & |s^{\alpha}D_{0^+}^\alpha s^{-\alpha} \sigma_s^{-1}R_s|\\
     =&\frac{1}{\Gamma(1-\alpha)}  s^{\alpha}\left|\left( \frac{ s^{-\alpha} \sigma_s^{-1}R_s}{s^\alpha} + \alpha \int_0^s \frac{ s^{-\alpha} \sigma_s^{-1}R_s -  r^{-\alpha} \sigma_r^{-1}R_r}{(s - r)^{\alpha+1}}  dr \right)\right|\\
     \leq & C\left| \frac{\varepsilon^2}{s^\alpha}+s^{-\alpha}\int_0^1 (1-x^{-\alpha})(1-x)^{\alpha+1}dx  \right. \\
     &\quad + \left.  \int_0^s \frac{r^{-\alpha} R_s(\sigma_s^{-1}-\sigma_r^{-1})}{(s-r)^{\alpha+1}}dr+\int_0^s \frac{r^{-\alpha}\sigma_r^{-1}(R_s-R_r)}{(s-r)^{\alpha+1}}ds   \right|\\
     \leq & C \varepsilon^2.
\end{align*}

By proceeding in the same way as in the case $H<\tfrac{1}{2}$, we obtain 
\[
\lim_{\varepsilon \to 0} \mathbb{E} \left[ \exp(C_3) \,\middle|\, \left\| \int_0^t \sigma_s\, dB^H_s \right\|_\beta \leq \varepsilon \right] = 1.
\]

Combining the above results, the functional takes the form  
\[
J(\phi,\dot{\phi}) = -\frac{1}{2} \int_0^1 \left(\dot{\phi}_s -s^{\alpha} D_{0^+}^\alpha s^{-\alpha} \sigma_s^{-1}  b_s(\phi_s)  \right)^2 +d_H\partial_x b_s(\phi_s)\, ds .
\]

\end{proof}

\subsection{Standard case }

In this section, we will prove the case $H=1/2$ .
By applying Itô's formula, we can extend the range of the H\"older norm considered in \cite{zhang2024onsagermachlupfunctionalstochasticdifferential}.

\begin{theorem}  \label{result3}
Let $X$ be the solution of \eqref{fir}, and assume that the coefficients satisfy Assumption \((A)\).
For any $\phi \in K_H^\sigma(L^2([0,1]))$ with $\phi_0 = x_0$, when $  H =1/2$, the Onsager--Machlup functional of $X$ with respect to the norms
$\Vert \cdot \Vert_\beta$ with $0<\beta<\frac{1}{2}-\frac{1}{2n}$, and $\Vert \cdot\Vert_\infty$ can be expressed as follows:
	\begin{equation*}
	J(\phi,\phi')=-\frac{1}{2} \int_0^1 \left(\frac{\phi_s' - b_s\left(\phi_s\right)}{\sigma_s}\right)^2+ \partial_xb_s(\phi_s)ds.
	\end{equation*}
    \end{theorem}
	\begin{proof}
   We only need to consider the limiting behavior of the stochastic integral term.  
By Taylor’s expansion, we have
\begin{align*}
	&\int_0^1 \sigma_s^{-1} b_s\!\left(\phi_s+\int_0^s \sigma_u \, dW_u\right) dW_s\\
	=& \int_0^1 \sigma_s^{-1} b_s(\phi_s)\, dW_s + \sum_{k=1}^{n-1} \frac{1}{k!}\int_0^1 \sigma_s^{-1}\partial_x^k b_s(\phi_s) 
	\left(\int_0^s \sigma_u \, dW_u\right)^k dW_s + \int_0^1 R_s \, dW_s .
\end{align*}

For $k \geq 2$, integration by parts yields
\begin{align*}
	&\int_0^1 \sigma_s^{-1}\partial_x^k b_s(\phi_s) 
	\left(\int_0^s \sigma_u \, dW_u\right)^k dW_s \\
	=& \int_0^1 \frac{\partial_x^k b_s(\phi_s)}{(k+1)\sigma_s^2} 
	   d\!\left(\int_0^s \sigma_u \, dW_u\right)^{k+1} - \int_0^1 \frac{k}{2}\, \partial_x^k b_s(\phi_s) 
	   \left(\int_0^s \sigma_u \, dW_u\right)^{k-1} ds \\
	=& \left. \frac{\partial_x^k b_s(\phi_s)}{(k+1)\sigma_s^2} 
	   \left(\int_0^s \sigma_u \, dW_u\right)^{k+1}\right|_{s=1}  - \int_0^1 \left(\int_0^s \sigma_u \, dW_u\right)^{k+1} 
	\frac{d}{ds}\!\left(\frac{\partial_x^k b_s(\phi_s)}{(k+1)\sigma_s^2}\right) ds \\
	&\quad - \int_0^1 \frac{k}{2}\, \partial_x^k b_s(\phi_s) 
	\left(\int_0^s \sigma_u \, dW_u\right)^{k-1} ds \\
    &\longrightarrow\; 0, \quad \varepsilon \to 0 .
\end{align*}

Moreover, the quadratic variation of the remainder satisfies
\begin{equation*}
	\left\langle \int_0^\cdot R_s \, dW_s \right\rangle_t \;\leq\; C \varepsilon^{2n},
\end{equation*}
and hence
\[
\mathbb{P}\!\Bigg(
   \Big|\int_0^1 R_s \, dW_s\Big| > \xi 
   \;\Big|\; \Big\|\int_0^{\cdot}\sigma_s \, dW_s\Big\|_{\beta} \leq \varepsilon
\Bigg)
\;\leq\;
\exp\!\Big(-C_1 \tfrac{\xi^2}{\varepsilon^{2n}} + C_2 \varepsilon^{-2/(1-2\beta)}\Big).
\]

Since $2n > 2/(1-2\beta)$, we require
\[
   \beta < \frac{1}{2} - \frac{1}{2n}.
\]
 Summing up, we obtain:
 \begin{equation*}
	J(\phi,\phi')=-\frac{1}{2} \int_0^1 \left(\frac{\phi_s' - b_s\left(\phi_s\right)}{\sigma_s}\right)^2+ \partial_xb_s(\phi_s)ds.
	\end{equation*}

\end{proof}
\section{Numerical experiments}
\label{Numeriacal experiments}	

In this section, we illustrate and validate our main theoretical results through numerical simulations of a specific stochastic differential equation (SDE). For a given SDE, we determine the most probable transition path from $X_0 = x_0$ to $X_1 = x_1$ using the Onsager–Machlup functional. To rigorously validate our results, we compare the obtained most probable path against simulated sample paths and their empirical mean trajectories.

The necessary condition for the most probable path is characterized by the following Euler–Lagrange equations:

\begin{itemize}
    \item For $1/4 < H < 1/2$:
    \begin{equation*}
    2\left( \frac{d}{ds} + \partial_x b_s(\phi_s^*) \right) 
    \left[ 
    \sigma_s^{-1} s^\alpha I^{\alpha}_{1^-} s^{-\alpha} \left((\phi_s^*)' -  s^{-\alpha}I_{0^+}^\alpha s^\alpha \frac{ b_s(\phi_s^*)}{\sigma_s} \right)
    \right] 
    = d_H \partial_{xx} b_s(\phi_s^*) ,
    \end{equation*}

    \item For $1/2<H < 1$:
    \begin{equation*}
    2\left( \frac{d}{ds} + \partial_x b_s(\phi_s^*) \right) 
    \left[ 
    \sigma_s^{-1} s^{-\alpha} D^{\alpha}_{1^-} s^{\alpha} \left( 
   (\phi_s^*)'-s^\alpha D_{0^+}^\alpha s^{-\alpha} \frac{  b_s(\phi_s^*) }{\sigma_s} \right)
    \right] 
    = d_H \partial_{xx} b_s(\phi_s^*) ,
    \end{equation*}
    \item For $H = 1/2$:
    \begin{equation*}
    2\left( \frac{d}{ds} + \partial_x b_s(\phi_s^*) \right) 
    \left[ 
    \frac{ (\phi_s^*)' - b_s(\phi_s^*) }{\sigma_s^2} 
    \right] 
    - \partial_{xx} b_s(\phi_s^*) = 0.
    \end{equation*}
\end{itemize}

All equations are subject to the boundary conditions:
\[
\phi_0^* = x_0, \quad \phi_1^* = x_1.
\]

We now consider an illustrative example: a double-well system driven by time-varying fractional noise.

\begin{example}
Consider the stochastic differential equation
\begin{equation}\label{example1}
dX_t = \frac{-X_t(X_t^2 - a^2)}{(1+X_t^2)^2}\,dt + c\bigl(2 + \sin(2n\pi t)\bigr)\,dB^H_t,
\end{equation}
where $a > 0$, $n \in \mathbb{Z}^+$, and $c \in \mathbb{R}\setminus\{0\}$.  
The coefficients of this equation satisfy our theoretical assumptions.

The corresponding noise-free system
\begin{equation}\label{example2}
   \frac{dX_t}{dt} = \frac{-X_t(X_t^2 - a^2)}{(1+X_t^2)^2}
\end{equation}
exhibits stable equilibrium states at $x = \pm a$ and an unstable equilibrium at $x = 0$.  
For system~\eqref{example2}, the solution starting from $-a$ remains at $-a$. 
However, when noise is introduced in system~\eqref{example1}, 
solutions starting from $-a$ can escape this stable equilibrium, 
with some sample paths reaching the other stable equilibrium at $a$. 

We investigate the most probable transition path of system~\eqref{example1} from $X_0 = -a$ to $X_t = a$. 
Fixing $a = 1$, we perform numerical simulations to plot the most probable paths, 
sample paths, and their mean trajectories for different values of $H$, $n$, and $c$. 

The drift term satisfies $|b(x)| = \left|\frac{-x(x^2-1)}{(1+x^2)^2}\right| < \infty$ and 
\begin{align*}
    |\partial_x b(x)| = \left| \frac{x^4-6x^2+1}{(1+x^2)^3} \right| \leq 1.
\end{align*}
It is straightforward to verify that Assumption~\((A)\) is satisfied when \(0 < H \leq \tfrac{1}{2}\). 
Moreover, for \(H = 0.6\) and \(\beta = 0.34\), we have
\begin{equation}
   \frac{(2\beta+1)\Gamma(1+\beta)^2}{9\alpha^2 \Gamma(\beta-\alpha)^2} \approx 1.015 > 1,
\end{equation}
which indicates that Assumption~\((A)\) is also fulfilled in this case.

Thus,  assumption \((A)\) are satisfied for $H = 0.3, 0.5,$ and $0.6$.  
We conduct numerical simulations for these three values of $H$ with parameter pairs $(c,n) = (1,1), (1/2,1),$ and $(1,4)$.

The simulations, illustrated in Figures~\ref{fig1}--\ref{fig3}, yield the following observations:
\begin{enumerate}
    \item The introduction of noise enables the system to transition between the two stable equilibria, a phenomenon that is precluded in the deterministic system.
    \item The smoothness of the sample paths varies significantly with the Hurst parameter $H$. Paths for $H=0.6$ are notably smoother, which is consistent with the higher regularity of the fractional Brownian motion for $H>1/2$. In contrast, paths for $H=0.3$ exhibit more erratic and sharp behavior, reflecting the rougher nature of the noise when $H<1/2$.
    \item The periodic modulation of the noise coefficient, through the $\sin(2n\pi t)$ term, is clearly imprinted on the dynamics. Both the sample paths and the most probable path display oscillations whose frequency correlates directly with the parameter $n$. This is particularly evident when comparing cases with $n=1$ and $n=4$.
    \item The empirical mean of the sample paths converges closer to the theoretically derived most probable path when the noise intensity coefficient $c$ is reduced from $1$ to $1/2$. This can be understood from the perspective of the OM action functional: any deviation from the most probable path $\phi^*$ incurs a higher ``action'' or ``energy cost''. In the small noise regime, paths with significantly higher action are exponentially suppressed. Therefore, the sample paths are statistically concentrated around the minimizer of the action, which is the most probable path.

\end{enumerate}
These results demonstrate that our theoretical framework effectively captures the essential features of the transition dynamics, including the influence of noise regularity (governed by $H$), noise intensity ($c$), and temporal structure.
\end{example}

\begin{figure}[htbp] 
    \centering

    \begin{subfigure}{\textwidth}
        \centering
        \includegraphics[height=0.3\textheight, keepaspectratio]{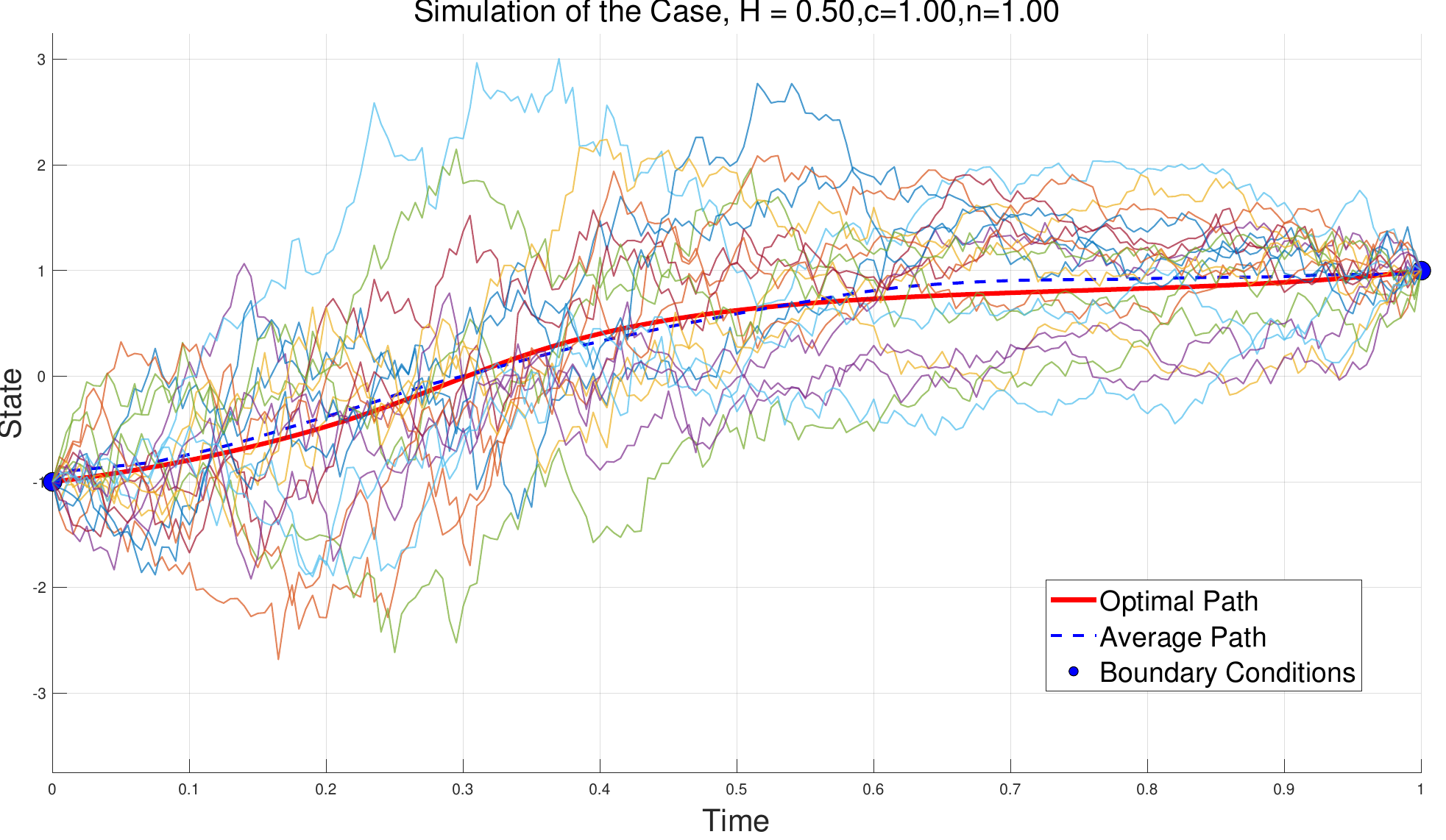}
        \caption{$H=1/2,c=1,n=1$}
    \end{subfigure}

    \vspace{0.2cm}

    \begin{subfigure}{\textwidth}
        \centering
        \includegraphics[height=0.3\textheight, keepaspectratio]{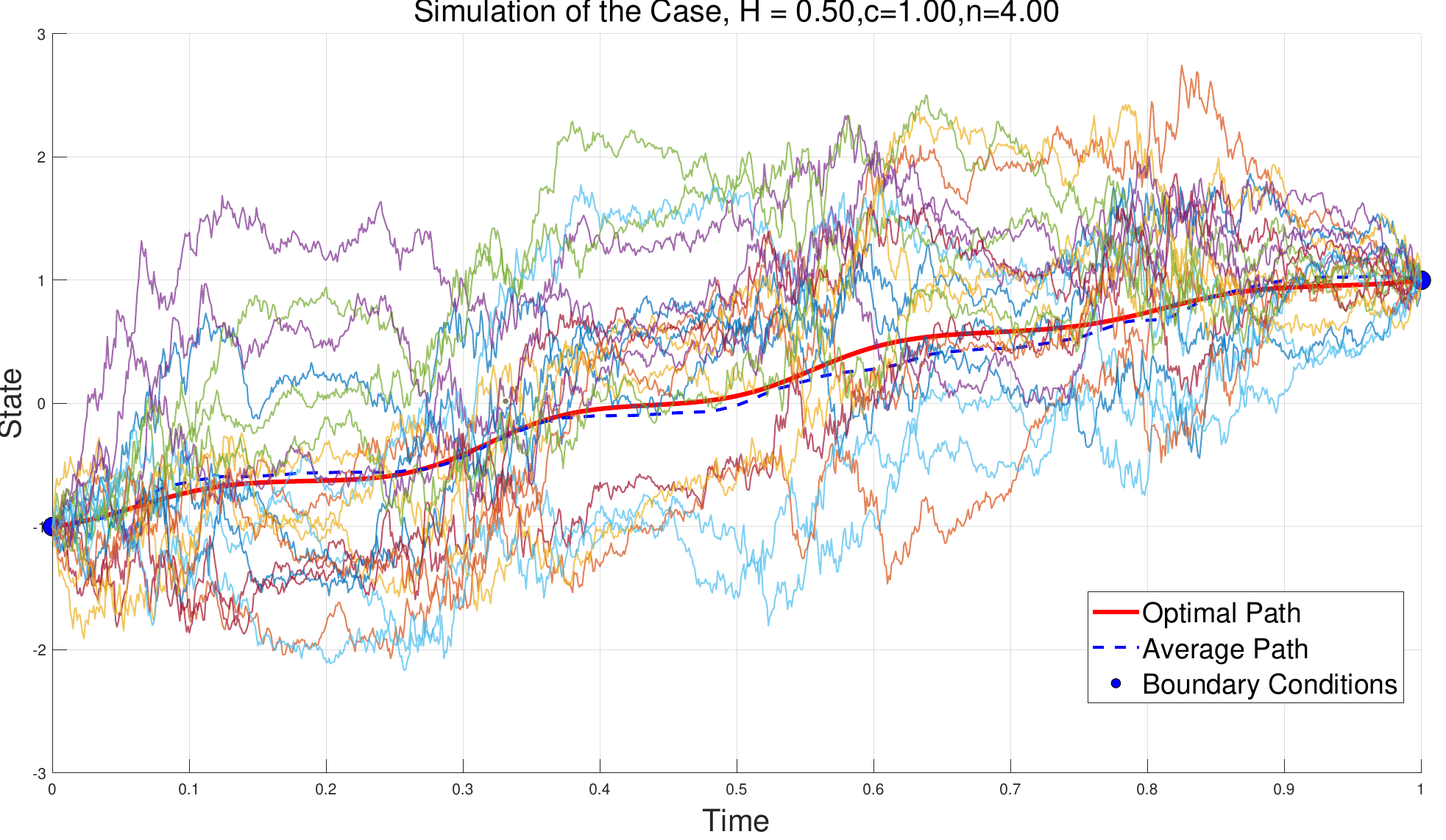}
        \caption{$H=1/2,c=1,n=4$}
    \end{subfigure}

    \vspace{0.2cm}

    \begin{subfigure}{\textwidth}
        \centering
        \includegraphics[height=0.3\textheight, keepaspectratio]{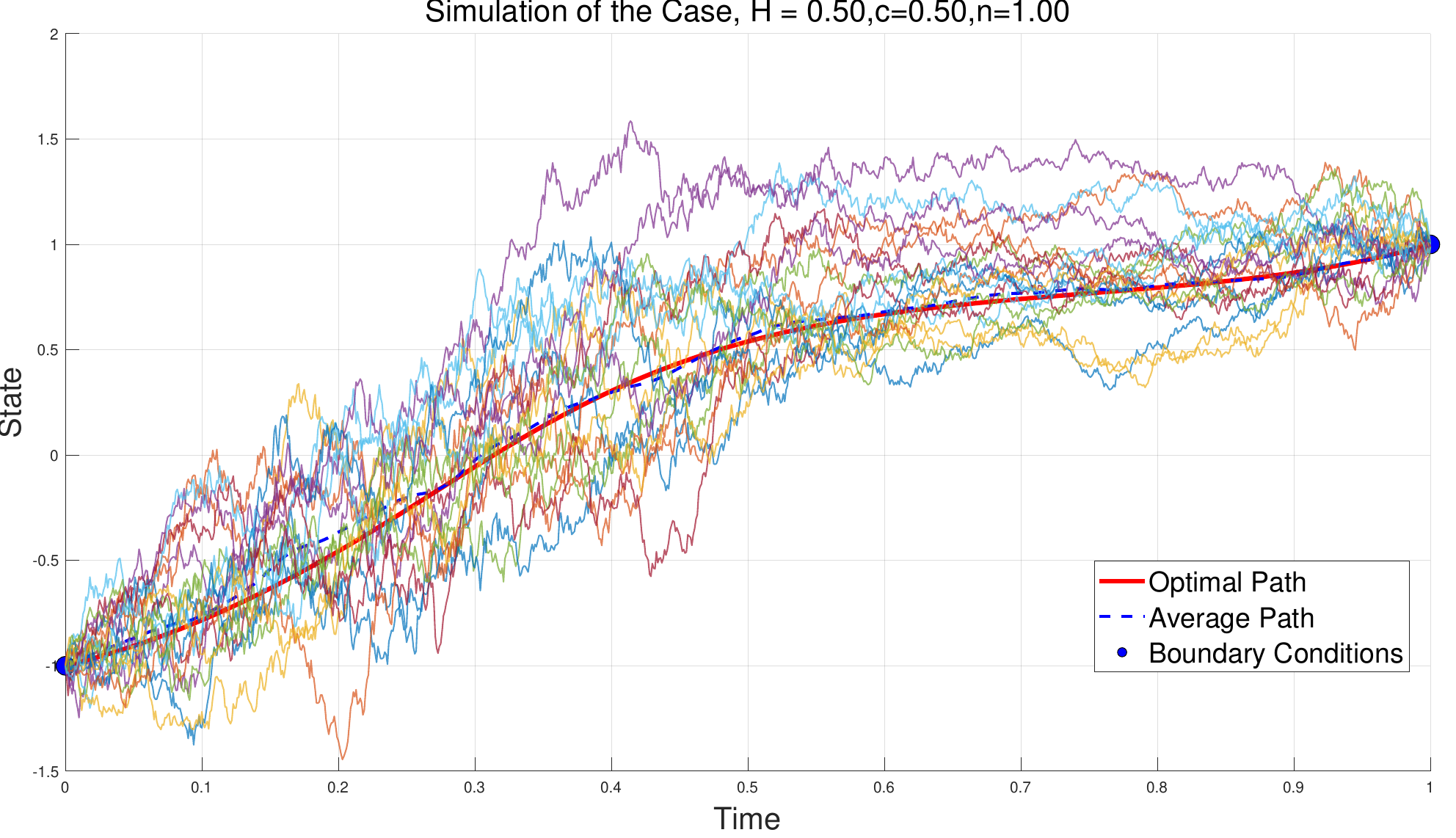}
        \caption{$H=1/2,c=1/2,n=1$}
    \end{subfigure}

    \caption{The standard cases}
    \label{fig1}
\end{figure}

\begin{figure}[htbp]
    \centering

    \begin{subfigure}{\textwidth}
        \centering
        \includegraphics[height=0.3\textheight, keepaspectratio]{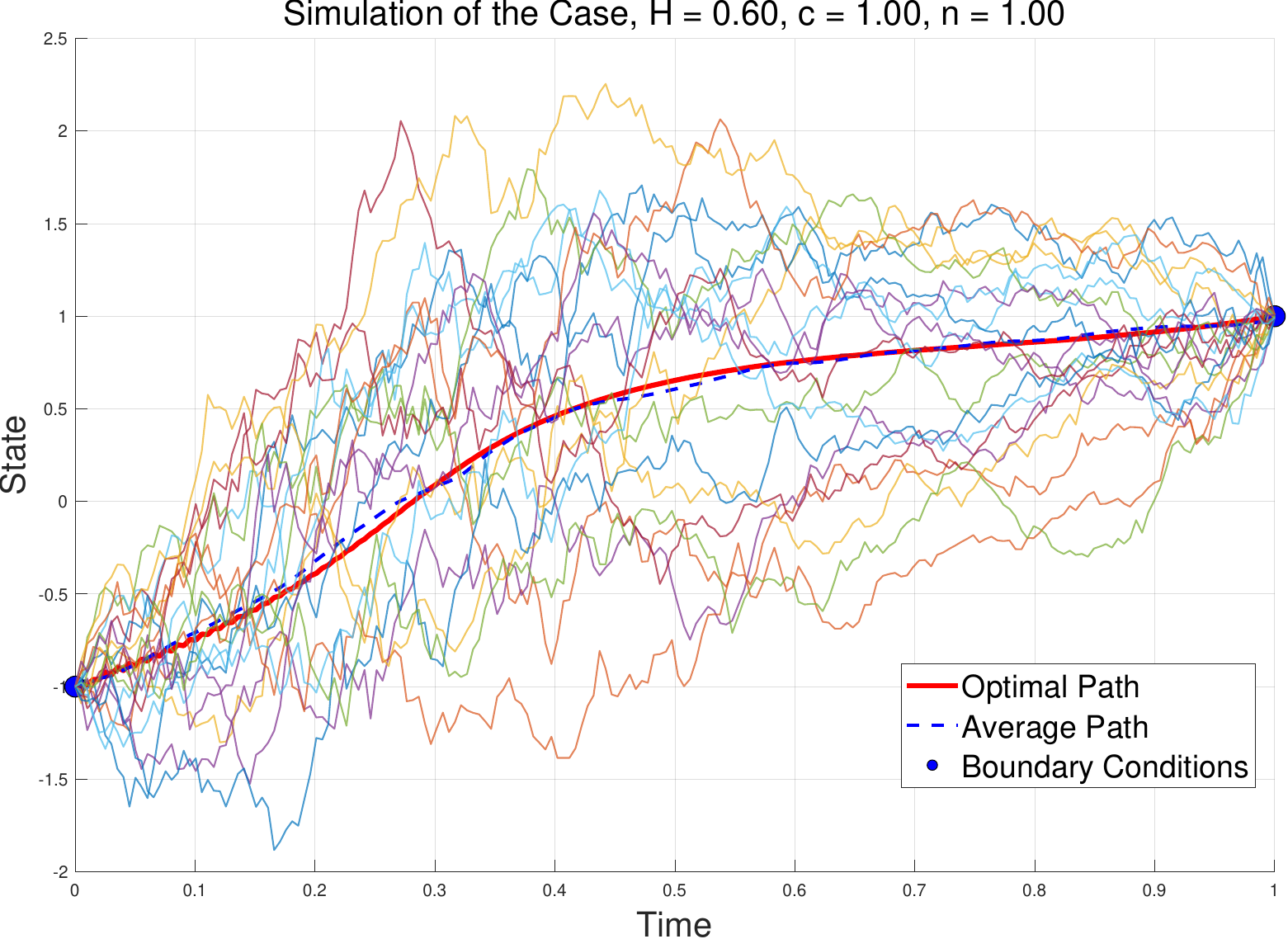}
        \caption{$H=3/5,c=1,n=1$}
    \end{subfigure}

    \vspace{0.2cm}

    \begin{subfigure}{\textwidth}
        \centering
        \includegraphics[height=0.3\textheight, keepaspectratio]{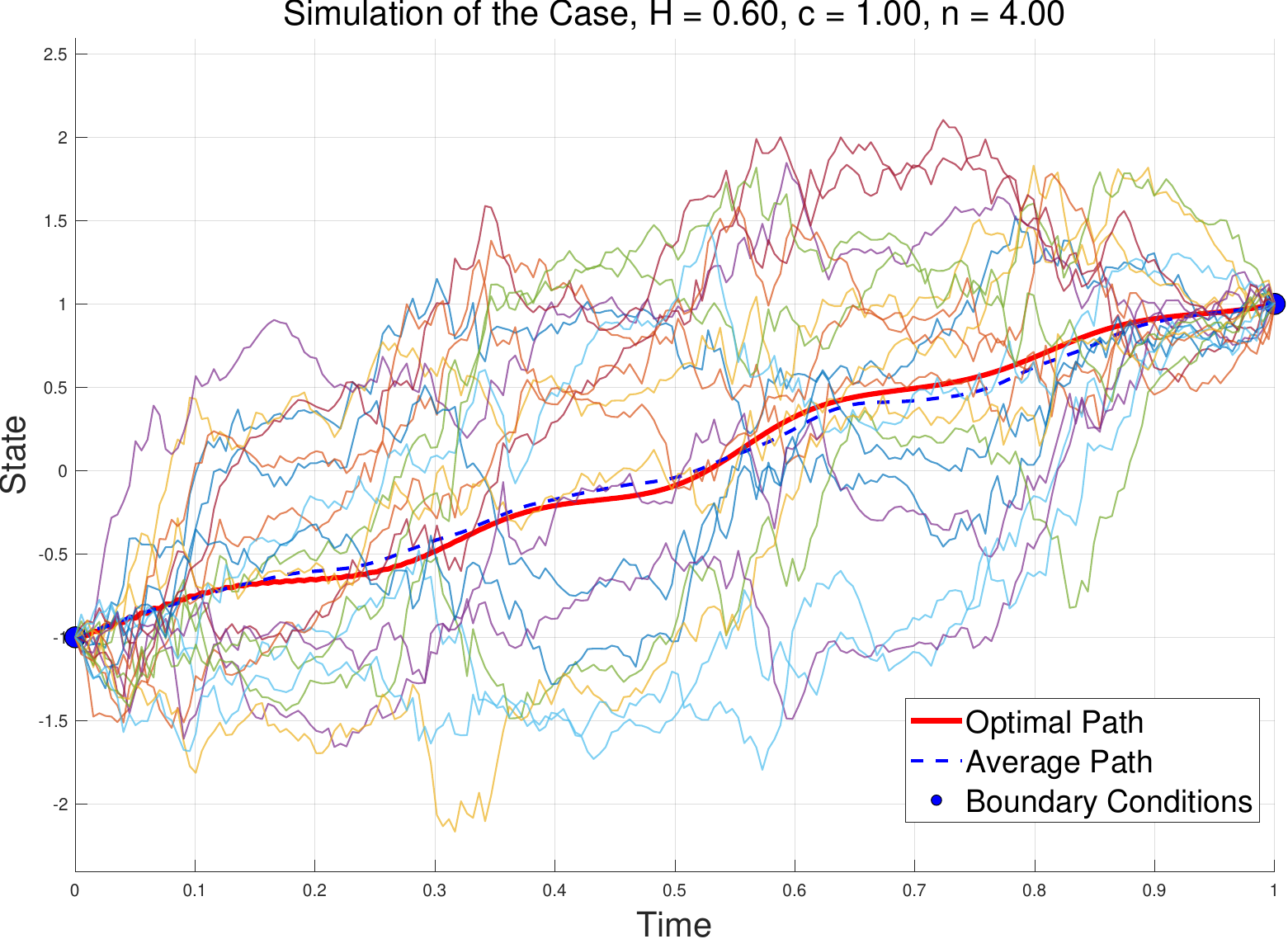}
        \caption{$H=3/5,c=1,n=4$}
    \end{subfigure}

    \vspace{0.2cm}

    \begin{subfigure}{\textwidth}
        \centering
        \includegraphics[height=0.3\textheight, keepaspectratio]{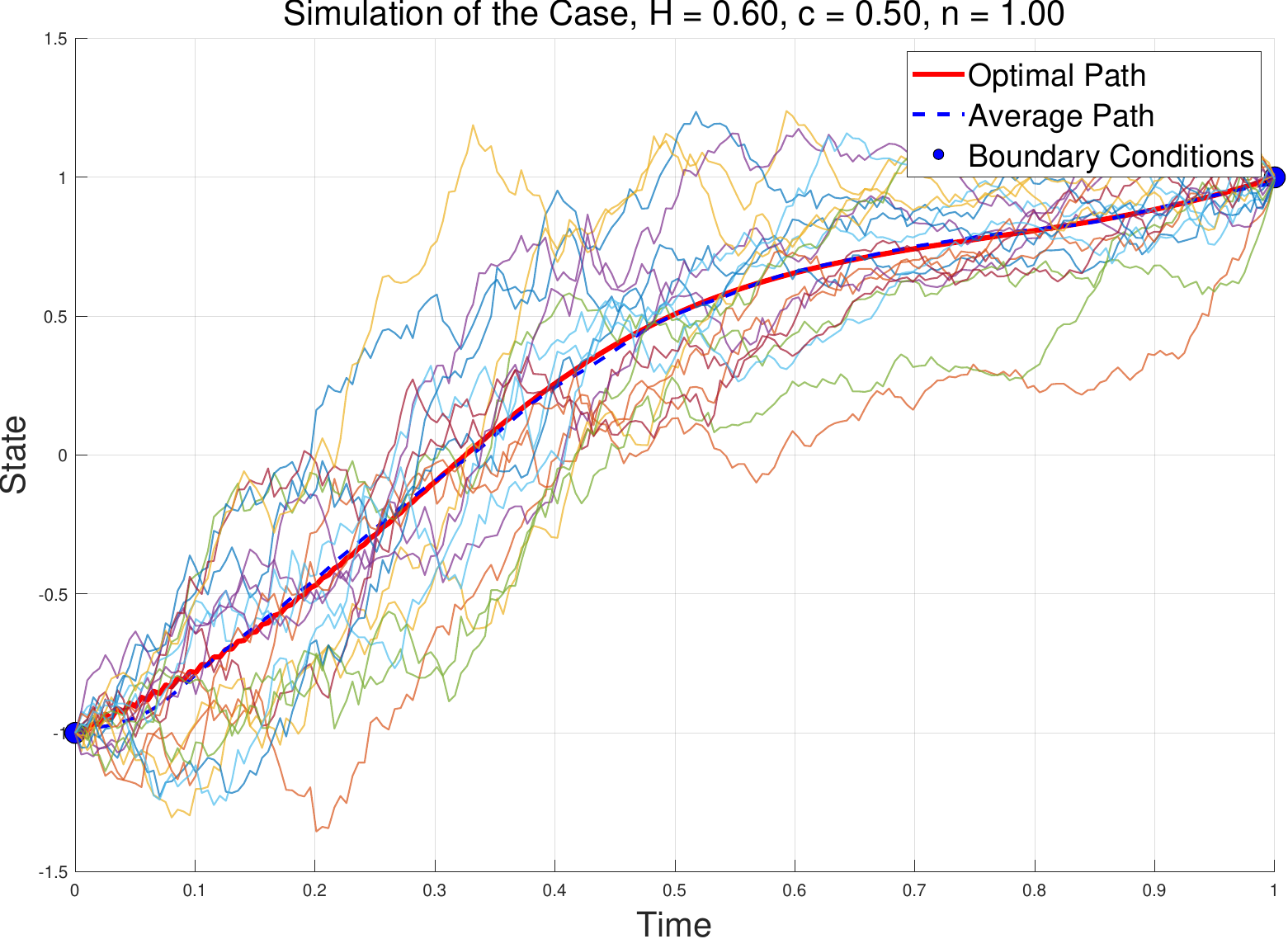}
        \caption{$H=3/5,c=1/2,n=1$}
    \end{subfigure}

    \caption{The regular cases}
    \label{fig2}
\end{figure}

\begin{figure}[htbp]\label{fig3}
    \centering

    \begin{subfigure}{\textwidth}
        \centering
        \includegraphics[height=0.3\textheight, keepaspectratio]{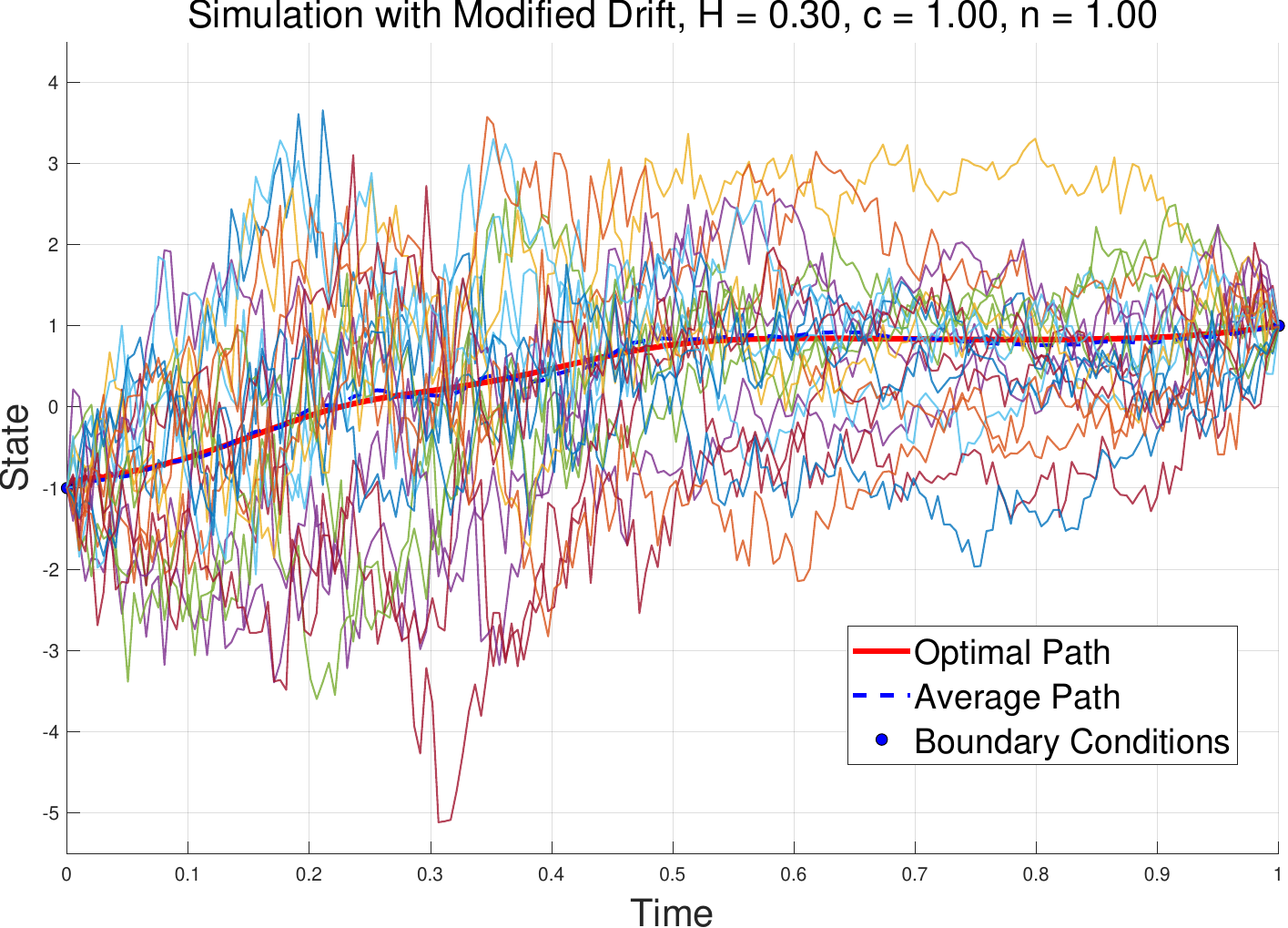}
        \caption{$H=3/10,c=1,n=1$}
    \end{subfigure}

    \vspace{0.2cm}

    \begin{subfigure}{\textwidth}
        \centering
        \includegraphics[height=0.3\textheight, keepaspectratio]{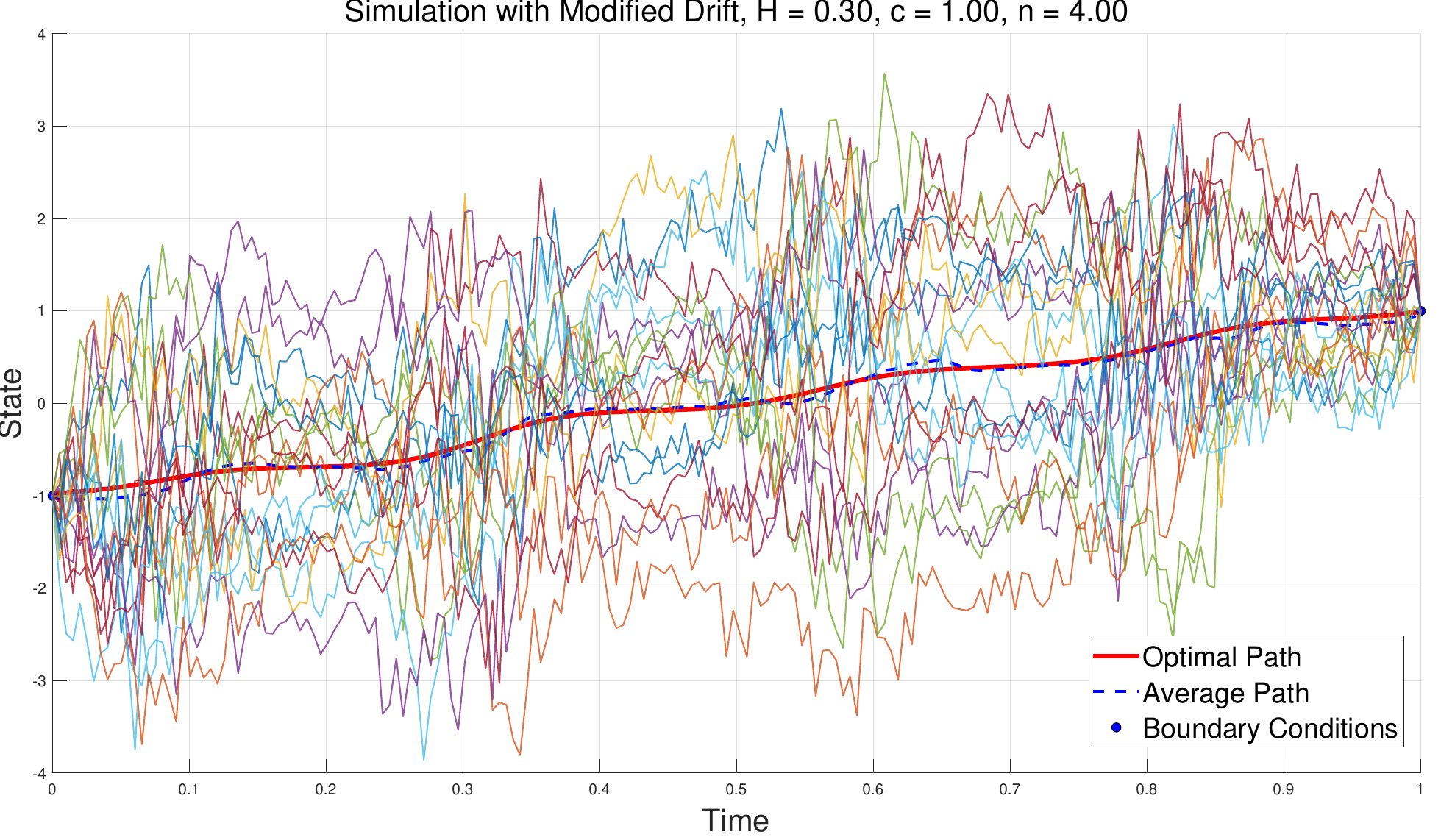}
        \caption{$H=3/10,c=1,n=4$}
    \end{subfigure}

    \vspace{0.2cm}

    \begin{subfigure}{\textwidth}
        \centering
        \includegraphics[height=0.3\textheight, keepaspectratio]{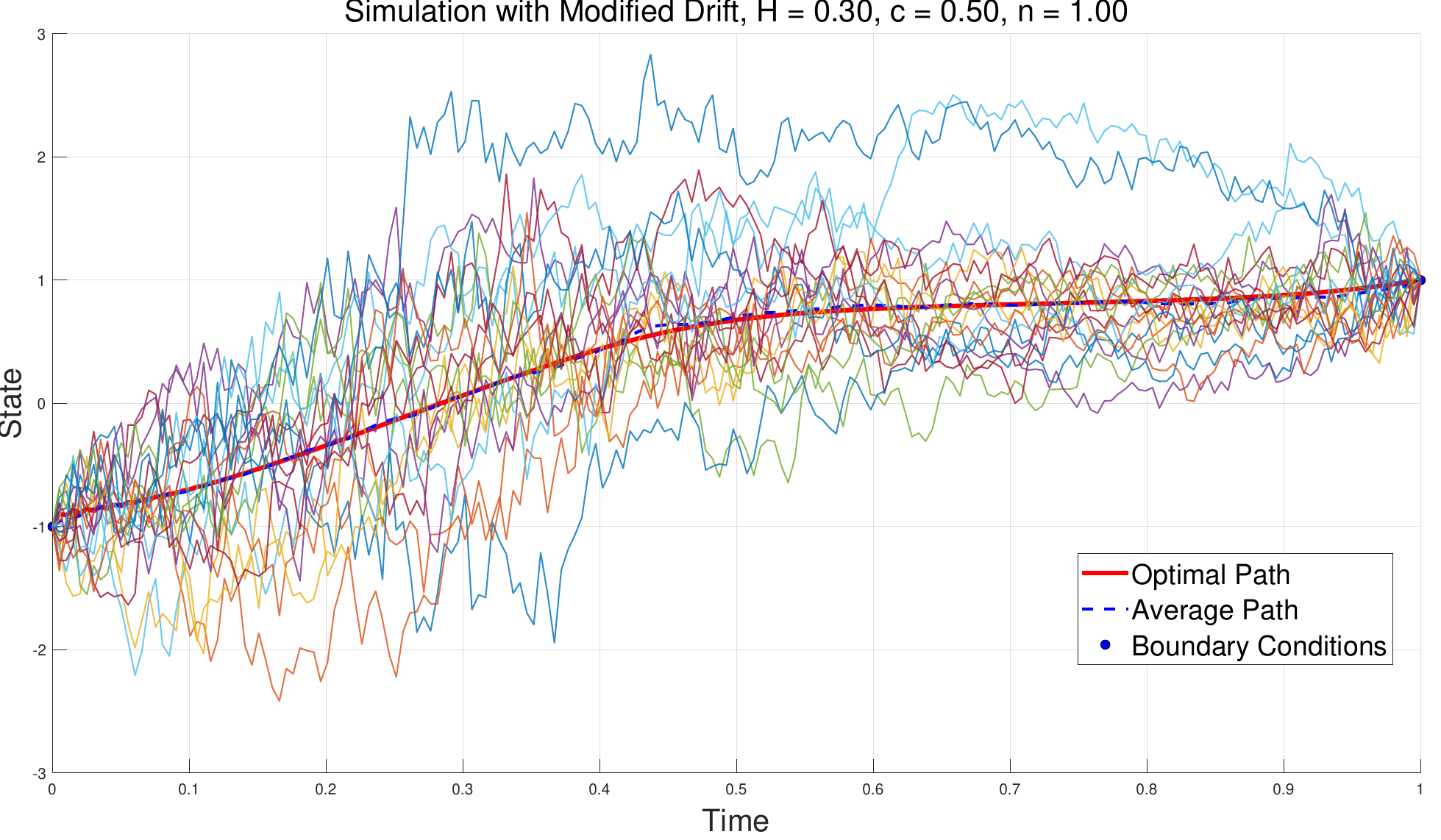}
        \caption{$H=3/10,c=1/2,n=1$}
    \end{subfigure}

    \caption{The singular cases}
    \label{fig3}
\end{figure}

\appendix
\renewcommand{\thesection}{\Alph{section}}
\section{Appendix}

\begin{lemma}[\cite{Nualart}]
     Suppose $H<1/2$. Then $K_H(t,s)$ has the following expression \begin{equation}
         K_H(r,u)=\frac{1}{\Gamma(\alpha)\Gamma(1-2\alpha)}u^{-\alpha}\int_u^r x^\alpha (x-1)^{\alpha-1}(r-x)^{-2\alpha}dx. \label{khdelta}
     \end{equation}
\end{lemma}

\begin{lemma}
 Assume $\tfrac{1}{4} < H < \tfrac{1}{2}$. Then the function $f$ defined in \eqref{fsu1} can be rewritten in the following form:
 \begin{align}
  \notag    &f(s,u)\\
      =&\int_u^s  \frac{1}{\Gamma(\alpha)} s^{-\alpha}  (s-r)^{\alpha-1} r^{\alpha} \sigma_r^{-1}\partial_x b_r(\phi_r)\left( \sigma_r  K_H(r,u) -\int_u^r \sigma_u' K_H(v,u)dv   \right)dr. \label{fsu}
 \end{align}
\end{lemma}
\begin{proof}
 By applying the integration by parts formula, we obtain
  \begin{align*}
    & \int_0^1 s^{-\alpha}I_{0^+}^\alpha s^{\alpha} \sigma_s^{-1}\partial_x b_s(\phi_s)\int_0^s \sigma_u dB^H_u    dW_s  \\
      =& \frac{1}{\Gamma(\alpha)}\int_0^1 s^{-\alpha} \int_0^s (s-r)^{\alpha-1} r^{\alpha} \sigma_r^{-1}\partial_x b_r(\phi_r)\int_0^r \sigma_u dB^H_u  dr dW_s \\
      =& \frac{1}{\Gamma(\alpha)}\int_0^1 s^{-\alpha} \int_0^s (s-r)^{\alpha-1} r^{\alpha} \sigma_r^{-1}\partial_x b_r(\phi_r)\left( \sigma_r B^H_r-\int_0^r \sigma_u' B^H_u du\right) drdW_s\\
       =& \frac{1}{\Gamma(\alpha)}\int_0^1 s^{-\alpha} \int_0^s (s-r)^{\alpha-1} r^{\alpha} \sigma_r^{-1}\partial_x b_r(\phi_r) \sigma_r \int_0^r K_H(r,u)dW_u drdW_s\\
       &\quad  - \frac{1}{\Gamma(\alpha)}\int_0^1 s^{-\alpha} \int_0^s (s-r)^{\alpha-1} r^{\alpha} \sigma_r^{-1}\partial_x b_r(\phi_r) \int_0^r \sigma_u'\int_0^u K_H(u,v)dv\,du\,drdW_s\\
       =&\int_0^1 dW_s\int_0^s dW_u \int_u^s dr  \frac{1}{\Gamma(\alpha)} s^{-\alpha}  (s-r)^{\alpha-1} r^{\alpha} \partial_x b_r(\phi_r)   K_H(r,u)\\
       &-\int_0^1 dW_s\int_0^s dW_v \int_v^s dr \int_v^r du
       \frac{1}{\Gamma(\alpha)} s^{-\alpha}  (s-r)^{\alpha-1} r^{\alpha} \sigma_r^{-1}\partial_x b_r(\phi_r) \sigma_u' K_H(u,v)\\
        =&\int_0^1 dW_s\int_0^s dW_u \int_u^s dr  \frac{1}{\Gamma(\alpha)} s^{-\alpha}  (s-r)^{\alpha-1} r^{\alpha} \partial_x b_r(\phi_r)  K_H(r,u)\\
       &-\int_0^1 dW_s\int_0^s dW_u \int_u^s dr \int_u^r dv
       \frac{1}{\Gamma(\alpha)} s^{-\alpha}  (s-r)^{\alpha-1} r^{\alpha} \sigma_r^{-1}\partial_x b_r(\phi_r) \sigma_v' K_H(v,u).
  \end{align*}
   Therefore,
   \begin{align*}
       f(s,u)&=\int_u^s dr  \frac{1}{\Gamma(\alpha)} s^{-\alpha}  (s-r)^{\alpha-1} r^{\alpha} \partial_x b_r(\phi_r)   K_H(r,u)\\
       &- \int_u^s dr \int_u^r dv
       \frac{1}{\Gamma(\alpha)} s^{-\alpha}  (s-r)^{\alpha-1} r^{\alpha} \sigma_r^{-1}\partial_x b_r(\phi_r) \sigma_v' K_H(v,u)\\
       &=f_1(s,u)+f_2(s,u),
   \end{align*}
   where
   \begin{equation}
       f_1(s,u)=\int_u^s dr  \frac{1}{\Gamma(\alpha)} s^{-\alpha}  (s-r)^{\alpha-1} r^{\alpha} \partial_x b_r(\phi_r)   K_H(r,u)   \label{apeq1}
   \end{equation}
   and
   \begin{equation}\label{apeq2}
       f_2(s,u)=- \int_u^s dr \int_u^r dv
       \frac{1}{\Gamma(\alpha)} s^{-\alpha}  (s-r)^{\alpha-1} r^{\alpha} \sigma_r^{-1}\partial_x b_r(\phi_r) \sigma_v' K_H(v,u).
   \end{equation}
\end{proof}

\begin{lemma}
   Let $\tfrac{1}{4} < H < \tfrac{1}{2}$, and let $f$ be the function defined in \eqref{fsu}. Then the Hilbert–Schmidt operator $K(\tilde{f})$ associated with the symmetrized function $\tilde{f}$ is a kernel operator.
\end{lemma}
\begin{proof} 
In order to prove 
\[
   \sum_{n=1}^\infty \big|\langle K(\tilde{f})e_n, e_n \rangle \big| < \infty,
\]
it suffices to show that
\[
   \sum_{n=1}^\infty \big|\langle K(f_i)e_n, e_n \rangle \big| < \infty, 
   \quad i=1,2,
\]
where $f_1$ and $f_2$ are defined in \eqref{apeq1} and \eqref{apeq2}, respectively.  
The proof is analogous to that of Lemma~13 in \cite{Nualart}.  

By \eqref{khdelta}, the term $f_2$ can be written as
\begin{align*}
    f_2(s,u)
    &= - \int_u^s \! dr \int_u^r \! dv \,
       \frac{1}{\Gamma(\alpha)} s^{-\alpha} (s-r)^{\alpha-1} r^\alpha 
       \sigma_r^{-1}\partial_x b_r(\phi_r) \sigma_v' K_H(v,u) \\
    &= -\frac{s^{-\alpha} u^{-\alpha}}{\Gamma(\alpha)^2 \Gamma(1-2\alpha)}
       \int_u^s \! dr \,(s-r)^{\alpha-1} r^\alpha 
       \sigma_r^{-1}\partial_x b_r(\phi_r) \\
    &\quad \times \int_u^r \! dv \, \sigma_v' 
       \int_u^v x^\alpha (x-u)^{\alpha-1}(v-x)^{-2\alpha}\, dx .
\end{align*}

Moreover, we have the representation
\begin{equation}\label{eq:frac-rep}
   (v-x)^{-2\alpha} 
   = \int_x^v (v-\xi)^{-1/2-\alpha+\delta}(\xi-x)^{-1/2-\alpha-\delta}\, d\xi, 
   \qquad |\delta|<H .
\end{equation}

By the expression of $K(f_2)$:
\begin{align*}
   &\sum_{n=1}^\infty \big|\langle K(f_2)e_n,e_n \rangle\big| \\
   =& \sum_{n=1}^\infty \Bigg|\int_0^1 ds \int_0^s du \int_u^s dr \int_u^r dv \int_u^v dx \int_x^v d\xi \,
       C s^{-\alpha} (s-r)^{\alpha-1} r^\alpha 
       \sigma_r^{-1}\partial_x b_r(\phi_r)\sigma_v' \\
   &\quad \qquad \times u^{-\alpha} x^\alpha (x-u)^{\alpha-1} 
       (v-\xi)^{-1/2-\alpha+\delta}(\xi-x)^{-1/2-\alpha-\delta} 
       e_n(s)e_n(u)\Bigg| \\
  \leq & \sum_{n=1}^\infty \Bigg|\int_0^1 d\xi \int_\xi^1 ds \int_\xi^s dr \int_\xi^r dv \int_0^\xi du \int_u^\xi dx \,
       C s^{-\alpha} (s-r)^{\alpha-1} r^\alpha 
       \sigma_r^{-1}\partial_x b_r(\phi_r)\sigma_v' \\
  & \quad\qquad \times u^{-\alpha} (x^\alpha-u^\alpha)(x-u)^{\alpha-1} 
       (v-\xi)^{-1/2-\alpha+\delta}(\xi-x)^{-1/2-\alpha-\delta} 
       e_n(s)e_n(u)\Bigg| .
\end{align*}
According to Cauchy's inequality, 
\begin{align*}
   &\sum_{n=1}^\infty \big|\langle K(f_2)e_n,e_n \rangle\big| \\
   &\leq C \int_0^1 d\xi 
     \Bigg\{\int_\xi^1 ds 
     \Bigg(\int_\xi^s dr \int_\xi^r dv \, 
       s^{-\alpha}(s-r)^{\alpha-1} r^\alpha 
       \sigma_r^{-1}\partial_x b_r(\phi_r)\sigma_v' 
       (v-\xi)^{-1/2-\alpha+\delta}\Bigg)^2 \Bigg\} \\
   &\quad \times \Bigg\{\int_0^\xi du 
     \Bigg(\int_u^\xi dx \, 
       u^{-\alpha} x^\alpha (x-u)^{\alpha-1} 
       (\xi-x)^{-1/2-\alpha-\delta}\Bigg)^2 \Bigg\}.
\end{align*}
Using elementary estimates and bounding $\sigma_r^{-1}\partial_x b_r(\phi_r)$ and $\sigma'_v$ uniformly, we obtain
\begin{align*}
   \sum_{n=1}^\infty \big|\langle K(f_2)e_n,e_n \rangle\big|
   &\leq C \int_0^1 d\xi 
      \Bigg\{\int_\xi^1 ds 
        \Bigg(\int_\xi^s dr \, s^{-\alpha}(s-r)^{\alpha-1}r^\alpha (r-\xi)^{1/2-\alpha-\delta}\Bigg)^2 \Bigg\} \\
   &\quad \times \Bigg\{\int_0^\xi du \, u^{-2\alpha} 
        \Bigg(\int_u^\xi dx \, x^\alpha (x-u)^{\alpha-1} (\xi-x)^{-1/2-\alpha-\delta}\Bigg)^2 \Bigg\}.
\end{align*}
Estimating each factor separately yields
\begin{align*}
   \sum_{n=1}^\infty \big|\langle K(f_2)e_n,e_n \rangle\big|
   &\leq C \int_0^1 d\xi 
        \Bigg\{\int_\xi^1 (s-\xi)^{1-2\delta}\, ds \Bigg\} 
        \Bigg\{\int_0^\xi u^{-2\alpha}\, \xi^{2\alpha}(\xi-u)^{-1-2\delta}\, du \Bigg\} \\
   &\leq C \int_0^1 (1-\xi)^{2-2\delta}\, \xi^{-2\delta}\, d\xi < \infty .
\end{align*}

Here the constant $C$ depends only on $H$, $\sup_{[0,1]}|\sigma'|$, and 
$\sup_{[0,1]}|\sigma_s^{-1}\partial_x b_s(\phi_s)|$.  
This proves the desired summability.

\end{proof}

\begin{lemma}
   Let $\tfrac{1}{2} < H < 1$, and define
   \begin{align*}
         f(s,u)&=\frac{c_H\Gamma(\alpha)}{\Gamma(1-\alpha)}	
	(
	 		s^{-\alpha}\sigma_s^{-1}\partial_x b(\phi_s)u^{-\alpha}I_{s^-}^\alpha u^\alpha  \sigma_u  \\
	 		&+\int_0^s \frac{\alpha \sigma_s^{-1}\partial_x b(\phi_s)u^{-\alpha}I_{s^-}^\alpha u^\alpha 
		\sigma_u}{(s-r)^{\alpha+1}}dr\\
		& -\int_u^s  \frac{\alpha s^\alpha r^{-\alpha} \sigma_r^{-1}\partial_x b(\phi_r) u^{-\alpha}I_{r^-}^\alpha u^\alpha  \sigma_u }{(s-r)^{\alpha+1}}dr .
   \end{align*}
    Then the Hilbert–Schmidt operator $K(\tilde{f})$ associated with the symmetrized function $\tilde{f}$ is a kernel operator.
  
\end{lemma}
\begin{proof}
    The proof can be established by following the argument in Lemma~14 of \cite{Nualart} . 
In our setting, it is sufficient to replace the kernel 
\[
K_H(s,u) = C u^{-\alpha} \int_u^s (x-u)^{\alpha-1} x^\alpha \, dx
\]
in Lemma~14 of  \cite{Nualart} with
\[
C u^{-\alpha} \int_u^s (x-u)^{\alpha-1} x^\alpha \sigma_x \, dx,
\]
and then observe that \(x^\alpha \sigma_x\) possesses the same H\"older continuity as \(x^\alpha\).

\end{proof}

\bibliographystyle{plain}
\bibliography{math.bib}

\end{document}